\documentclass[11pt]{article}

\usepackage{amsfonts}
\usepackage{amssymb,amsmath,amsthm}
\usepackage{latexsym}
\usepackage{fullpage}
\usepackage{hyperref}
\usepackage{graphicx}
\usepackage{tkz-euclide,subfigure}
\usepackage{tikz}
\usepackage[utf8]{inputenc}
\usepackage{pgfplots}
\usetikzlibrary{shapes.misc}
\usepackage{ amssymb }

\newtheorem{theorem}{Theorem}[section]

\newtheorem{prop}[theorem]{Proposition}

\newtheorem{lemma}[theorem]{Lemma}
\newtheorem{cor}[theorem]{Corollary}
\newtheorem{definition}[theorem]{Definition}

\theoremstyle{definition}
\newtheorem{remark}[theorem]{Remark}

\newcounter{tenumerate}

\def\P{\mathbb{P}}

\newcommand*\circled[1]{\tikz[baseline=(char.base)]{
            \node[shape=circle,draw,inner sep=2pt](char){#1};}}

\newcommand{\one}{\1}
\newcommand{\deq}{\stackrel{\scriptscriptstyle\triangle}{=}}

\renewcommand{\epsilon}{\varepsilon}

\newcommand{\1}{\mathbf{1}}

\DeclareMathOperator{\var}{Var}

\newcommand{\E}{{\mathbb E}}
\newcommand{\remove}[1]{}

\renewcommand{\leq}{\leqslant}
\renewcommand{\geq}{\geqslant}

\def\XXint#1#2#3{{\setbox0=\hbox{$#1{#2#3}{\int}$}
\vcenter{\hbox{$#2#3$}}\kern-.5\wd0}}

\def\plus{\mathsf{+}}
\def\minus{\mathsf{-}}
\def\zero{\mathsf{0}}

\title{Exponential decay of correlations in the two-dimensional random field Ising model}
\author{Jian Ding\thanks{Partially supported by NSF grant DMS-1757479 and an Alfred Sloan fellowship.}  \\ University of Pennsylvania \and Jiaming Xia\footnotemark[1]  \\ University of Pennsylvania}

\begin{document}

\maketitle

\begin{abstract}
We study random field Ising model on $\mathbb Z^2$ where the external field is given by i.i.d.\ Gaussian variables with mean zero and positive variance.  We show that  the effect of boundary conditions on the magnetization in a finite box decays exponentially in the distance to the boundary.
\end{abstract}

\section{Introduction}

For $v\in \mathbb Z^2$, let $h_v$ be i.i.d.\ Gaussian variables with mean zero and variance $\epsilon^2>0$.
We consider random field Ising model (RFIM) with external field $\{h_v: v\in \mathbb Z^2\}$ at temperature $T = 1/\beta \in [0, \infty)$.
For $N \geq 1$, let $\Lambda_N = \{v\in \mathbb Z^2: |v|_\infty \leq N\}$ be a box in $\mathbb Z^2$ centered at the origin $o$ and of side length $2N$. For any set $A\subset \mathbb Z^2$, define $\partial A = \{v\in \mathbb Z^2\setminus A: u\sim v\mbox{ for some } u\in  A\}$ (where $u\sim v$ if $|u-v|_1 =1$). The RFIM Hamiltonian $H^{\Lambda_N, \pm}$ on the configuration space $\{-1, 1\}^{\Lambda_N}$ with plus (respectively, minus) boundary condition and external field $\{h_v: v\in \Lambda_N\}$ is defined to be
\begin{equation}\label{eq-def-H}
H^{\Lambda_N,\pm}(\sigma) = - \big(\sum_{u\sim v, u, v\in \Lambda_N} \sigma_u \sigma_v  \pm \sum_{u\sim v, u\in  \Lambda_N, v\in \partial \Lambda_N} \sigma_u + \sum_{u\in \Lambda_N} \sigma_u h_u\big) \mbox{ for } \sigma \in \{-1, 1\}^{\Lambda_N}\,.
\end{equation}
(In the preceding summation, each unordered pair $u\sim v$ only appears once.)
Quenched on the external field $\{h_v\}$, the Ising measure with plus boundary condition (respectively minus boundary condition) is defined such that for all $\sigma\in \{-1, 1\}^{\Lambda_N}$ (throughout the paper the temperature is fixed, and thus we suppress the dependence on $\beta$ in all notations)
\begin{equation}
\label{eq-def-mu}
\mu^{\Lambda_N, \pm}(\sigma)=\frac{e^{-\beta H^{\Lambda_N, \pm}(\sigma)}}{Z^{\Lambda_N, \pm}}, \mbox{ where } Z^{\Lambda_N,\pm}=\sum_{\sigma'\in\{-1,1\}^{\Lambda_N}}e^{-\beta H^{\Lambda_N, \pm}(\sigma')}.
\end{equation}
Note that $\mu^{\Lambda_N, \pm}$ is a random measure which itself depends on $\{h_v\}$. To be clear of the two different sources of randomness, we use $\P$ and $\E$ to refer to the probability measure with respect to the external field $\{h_v\}$; and we use $\mu^{\Lambda_N, \pm}$ for the Ising measures and use $\langle \cdot \rangle_{\mu^{\Lambda_N, \pm}}$ to denote the expectations with respect to the Ising measures.

\begin{theorem}\label{thm-main}
For any $\epsilon >0,  T\in [0, \infty)$, there exists $c = c(\epsilon, T)>0$ such that
$$\E(\langle \sigma_o \rangle_{\mu^{\Lambda_N, +}} - \langle \sigma_o \rangle_{\mu^{\Lambda_N,-}}) \leq c^{-1} e^{-c N} \mbox{ for all } N\geq 1\,.$$
\end{theorem}

This result lies under the umbrella of the general Imry--Ma \cite{IM75} phenomenon, which states that in two-dimensional systems any first order transition is rounded off upon the introduction of arbitrarily weak static, or quenched,
disorder in the parameter conjugate to the corresponding extensive quantity. In the particular case of RFIM, it was shown in \cite{AW89, AW90} that \emph{for all non-negative temperatures} the effect on the local quenched magnetization of the boundary conditions
at distance $N$ away decays to 0 as $N\to \infty$, which also implies the uniqueness of the Gibbs state. The decay rate was then improved to $1/\log\log N$ in \cite{Chatterjee18} and to $1/N^\gamma$ (for some $\gamma>0$) in \cite{AP18}. All these results apply for arbitrarily weak disorder.  In the presence of high disorder it has been shown that there is an exponential decay \cite{Ber85, FI84, CJN18} (see also \cite[Appendix A]{AP18}). The main remaining challenge is to decide whether the decay rate is exponential when the disorder is weak. In fact,  there have been debates even among physicists on whether there exists a regime where the decay rate is polynomial, and weak supporting arguments have been made in both directions \cite{GMS82, BK87, DS84} --- in particular in \cite{DS84} an argument was made for polynomial decay \emph{at zero temperature} for a certain choice of disorder. Theorem~\ref{thm-main} provides a complete answer to this question when the random field consists of i.i.d.\ Gaussian variables.

The two-dimensional behavior of RFIM is drastically different from that for dimensions three and higher: it was shown in \cite{Imbrie85} that at zero temperature
the effect on the local quenched magnetization of the boundary conditions at distance $N$ does not vanish in $N$ in the presence of weak disorder, and later an analogous result was proved  in \cite{BK88} at low temperatures. A heuristic explanation behind the different behaviors is as follows: in dimension two the fluctuation of the random field in a box is of the same order as the size of the boundary, while in dimensions three and above the fluctuation of the random field is substantially smaller than the size of the boundary.

Our proof method is different from all of \cite{AW90, Chatterjee18, AP18} (and different from \cite{Ber85, FI84, CJN18}), except that in the heuristic level our proof seems to be related to the Mandelbrot percolation analogy presented in \cite[Appendix B]{AP18}. The works \cite{AW89, AW90} treated a wide class of distributions for disorder, while \cite{Chatterjee18, AP18} and this paper work with Gaussian disorder. The main features of Gaussian distributions used in this paper are the simple formula for the change of measure (see \eqref{eq-zero-change-of-measure}) and linear decompositions for Gaussian process (see \eqref{eq-zero-Gaussian-conditioning}). In addition, we remark that the analysis in \cite{AW90, AP18} extends to the case with finite-range interactions. While we expect our framework to be useful in analyzing the finite-range case, the lack of planar duality seems to present some non-trivial obstacle.

The rest of the paper consists of two sections. In Section~\ref{sec:zero-temperature}, we prove Theorem~\ref{thm-main} in the special case of $T = 0$. In our opinion, this is a significant simplification of the general case but still captures the core challenge of the problem. We hope that some of the key ideas (e.g., the crucial application of \cite{AB99}) can be more transparent by first presenting the proof in this simplified case. In Section~\ref{sec:positive-temperature}, we then present the proof for the case of $T>0$. While the proof naturally shares the key insights with the case for $T = 0$, it seems to us that there are significant additional obstacles. As a result, the proof is not presented as an extension of the zero-temperature case. Instead, we present an almost self-contained proof, but omit details at times when they are merely adaption of arguments in Section~\ref{sec:zero-temperature}.

Our (shared) notations in Sections~\ref{sec:zero-temperature} and \ref{sec:positive-temperature} are consistent with each other, and a few notations in Section~\ref{sec:positive-temperature} are natural extensions of those in Section~\ref{sec:zero-temperature}. However, for clarity of exposition, we will recall or re-explain all notations in Section~\ref{sec:positive-temperature}.

\noindent {\bf Concurrent work.} During the submission of this paper, a paper \cite{AHP19} which proved the same result was completed. The proof of \cite{AHP19} was inspired by the proof at zero temperature in this paper (for the crucial application of \cite{AB99}). Both proofs share the basic intuition of ``using the fluctuation of the sum of the random field in a box to fight the influence of the boundary condition'' (which went back to \cite{AW89, AW90}) and both apply \cite{AB99} to disagreement percolation in a crucial manner. However, the two approaches seem to be rather different in at least the following two important aspects: (1) This paper employs first moment analysis via various perturbations of the random field, and the paper \cite{AHP19} (similar to \cite{AP18}) relies on concentration/anti-concentration type of analysis (which in particular uses second-moment computations); (2) At positive temperatures, this paper employs a certain monotone coupling (adaptive admissible coupling as in Definition~\ref{def-adaptive-admissible-coupling}) between Ising measures with different boundary conditions, and the paper \cite{AHP19} considers a continuous extension of the Ising model into the metric graph which allows to study spin correlations via disagreement percolation for two \emph{independent} samples (inspired by \cite{Sheffield05, vdBerg93}).

\section{Exponential decay at zero temperature}\label{sec:zero-temperature}

At  zero temperature, $\mu^{\Lambda_N, +}$ (and respectively $\mu^{\Lambda_N, -}$) is supported on the minimizer of \eqref{eq-def-H}, which is  known as the \emph{ground state} and is unique with probability 1. We denote by $\sigma^{\Lambda_N, +}$ the ground state with respect to the plus-boundary condition and by $\sigma^{\Lambda_N,-}$ the ground state with respect to the minus-boundary condition. Therefore, for $T=0$ we have the simplification that the only randomness is from the $\P$-measure. Thus, Theorem~\ref{thm-main} for $T=0$ can then be simplified as follows.
\begin{theorem}\label{thm-main-zero}
For any $\epsilon>0$, there exists $c = c(\epsilon)>0$ such that  $\P(\sigma^{\Lambda_N, +}_o \neq \sigma^{\Lambda_N, -}_o) \leq c^{-1} e^{-c N}$ for all $N\geq 1$.
\end{theorem}

\subsection{Outline of the proof}

We first reformulate Theorem~\ref{thm-main-zero}.
For $v\in  \Lambda_N$, we define
\begin{equation}\label{eq-zero-def-xi}
\xi_v^{\Lambda_N} =
\begin{cases}
\plus, & \mbox{ if } \sigma^{\Lambda_N, +}_v = \sigma^{\Lambda_N, -}_v = 1\,,\\
\minus, & \mbox{ if } \sigma^{\Lambda_N, +}_v = \sigma^{\Lambda_N, -}_v = -1\,,\\
\zero, & \mbox{ if }   \sigma^{\Lambda_N, +}_v = 1 \mbox{ and } \sigma^{\Lambda_N, -}_v = -1\,.
\end{cases}
\end{equation}
By monotonicity (c.f. \cite[Section 2.2]{AP18}), the case of  $\sigma^{\Lambda_N, +}_v = -1 \mbox{ and } \sigma^{\Lambda_N, -}_v = 1$ cannot occur, so $\xi_v^{\Lambda_N}$ is well-defined for all $v\in \Lambda_N$.  Theorem~\ref{thm-main-zero} can be restated as
\begin{equation}\label{eq-zero-main-result}
m_N \leq c^{-1} e^{-c N} \mbox{ for } c = c(\epsilon)>0, \mbox{ where } m_N \deq \P(\xi_{o}^{\Lambda_N}=\zero) \,.
\end{equation}
For any $A\subset \mathbb Z^2$,  we can analogously define $\xi^A$ by replacing $\Lambda_N$ with $A$ in \eqref{eq-def-H} and \eqref{eq-zero-def-xi}. Let $\mathcal C^{A} = \{v\in A: \xi_v^{A} = \zero\}$ (that is, $\mathcal C^A$ is the collection of \emph{disagreements}).
Monotonicity (see \cite[(2.7)]{AP18}) implies that
\begin{equation}\label{eq-zero-monotonicity}
\mathcal C^{B} \cap B' \subset \mathcal C^{B'} \mbox{ provided that } B' \subset B\,.
\end{equation}
In particular, this implies that $m_N$ is decreasing in $N$, so we need only consider $N=2^n$ for $n\geq 1$.
Clearly, for any $v\in \mathcal C^A$, there exists a path in $\mathcal C^A$ joining $v$ and $\partial A$. This suggests consideration of percolation properties of $\mathcal C^A$. Indeed, a key step in our proof for \eqref{eq-zero-main-result} is the following proposition on the lower bound on the length exponent for geodesics (i.e., shortest paths) in $\mathcal C^{\Lambda_N}$. For any $A\subset \mathbb Z^2$, we denote by $d_A (\cdot, \cdot)$ the intrinsic distance on $A$, i.e., the graph distance on the induced subgraph on $A$. Let $d_A (A_1, A_2) = \min_{x\in A_1 \cap A, y\in A_2 \cap A} d_A(x, y)$ (with the convention that $\min \emptyset = \infty$).
\begin{prop}\label{prop-zero-crossing-dimension}
There exist $\alpha = \alpha(\epsilon) > 1$, $\kappa = \kappa(\epsilon)>0$ such that for all $N\geq 1$
\begin{equation}\label{eq-zero-box-counting}
\P(d_{\mathcal C^{\Lambda_N}} (\partial \Lambda_{N/4}, \partial \Lambda_{N/2}) \leq N^{\alpha}) \leq \kappa^{-1} e^{- N^{\kappa}}\,.
\end{equation}
\end{prop}
The proof of Proposition~\ref{prop-zero-crossing-dimension} will rely on \cite{AB99}, which takes  the next lemma as input. For any rectangle $A\subset \mathbb R^2$ (whose sides are not necessarily parallel to the axes), let $\ell_A$ be the length of the longer side and let $ A^{\mathrm{Large}}$ be (the lattice points of) the square box concentric with $A$, of side length $32\ell_A$ and with sides parallel to axes. In addition, define the aspect ratio of $A$ to be the ratio between the lengths of the longer and shorter sides. For a (random) set $\mathcal C \subset \mathbb Z^2$, we use $\mathrm{Cross}(A, \mathcal C)$ to denote the event that there exists a path $v_0, \ldots, v_k \in A \cap \mathcal C$ connecting the two shorter sides of $A$ (that is,  $v_0, v_k$ are of $\ell_\infty$-distances less than 1 respectively from the two shorter sides of $A$).
\begin{lemma}\label{lem-zero-assumption}
Write $a = 100$. There exists $\ell_0 = \ell_0(\epsilon)$ and $\delta = \delta(\epsilon) >0$ such that the following holds for any $N\geq 1$. For any $k\geq 1$ and any rectangles $A_1, \ldots, A_k \subseteq \{v\in \mathbb R^2: |v|_\infty \leq N/2\}$ with aspect ratios at least $a$ such that  (a) $\ell_0\leq \ell_{A_i} \leq N/32$ for all $1\leq i\leq k$ and (b) $A^{\mathrm{Large}}_1, \ldots, A^{\mathrm{Large}}_k$ are disjoint, we have
$$\P(\cap_{i=1}^k \mathrm{Cross}(A_i, \mathcal C^{\Lambda_N})) \leq (1-\delta)^k\,.$$
\end{lemma}
(Actually, the authors of \cite{AB99} treated random curves in $\mathbb R^2$. However, the main capacity analysis can be copied in the discrete case, and the connection between the capacity and the box-counting dimension is straightforward (c.f. \cite[Lemma 2.3]{DT16}).) Armed with Lemma~\ref{lem-zero-assumption}, we can apply \cite[Theorem 1.3]{AB99} to deduce that for some $\alpha = \alpha(\epsilon)>1$,
\begin{equation}\label{eq-zero-box-counting-weak}
\P(d_{\mathcal C^{\Lambda_N}} (\partial \Lambda_{N/4}, \partial \Lambda_{N/2}) \leq N^{\alpha}) \to 0 \mbox{ as } N \to \infty\,.
\end{equation}
By a standard percolation argument (Lemma~\ref{lem-zero-enhance}) which we will explain later, we can enhance the probability decay in \eqref{eq-zero-box-counting-weak} and prove \eqref{eq-zero-box-counting}.

By \eqref{eq-zero-monotonicity},
the random set
$\mathcal C^{\Lambda_N} \cap A$ is stochastically dominated by  $\mathcal C^{A^{\mathrm{Large}}} \cap A$ as long as $A^{\mathrm{Large}}\subset \Lambda_N$. Moreover, it is obvious that $\mathcal C^{A^{\mathrm{Large}}_i}$ for $1\leq i\leq k$ are mutually independent, as long as the sets $A_i^{\mathrm{Large}}$ for $1\leq i\leq k$ are disjoint. Therefore, in order to prove Lemma~\ref{lem-zero-assumption}, it suffices to show that for any rectangle $A$ with aspect ratio at least $a = 100$ we have

\begin{equation}\label{eq-zero-crossing-prob}
\P(\mathrm{Cross}(A, \mathcal C^{A^{\mathrm{Large}}})) \leq 1- \delta \mbox{ where } \delta=  \delta(\epsilon) >0\,.
\end{equation}

Both the proof of \eqref{eq-zero-crossing-prob} and the application of \eqref{eq-zero-box-counting} rely on a perturbative analysis, which is another key feature of our proof. Roughly speaking, the logic is as follows:
\begin{itemize}
\item We first consider the perturbation by increasing the field by an amount of order $1/N$, and use this to show that the probability for a $\zero$-valued contour surrounding an annulus is strictly bounded away from 1.
\item Based on this property, we prove  \eqref{eq-zero-crossing-prob}, which then implies \eqref{eq-zero-box-counting}.
\item Given \eqref{eq-zero-box-counting}, we then show that increasing the field by an amount of order $1/N^\alpha$ (recall that $\alpha>1$ is from Proposition~\ref{prop-zero-crossing-dimension} and thus the perturbation here is $1/N^\alpha \ll 1/N$) will most likely change the $\zero$'s to $\plus$'s. Based on this, we prove polynomial decay for $m_N$ with large power, which can then be enhanced to exponential decay.
\end{itemize}
For compactness of exposition, the actual implementation will differ slightly from the above plan:
\begin{itemize}
\item We first prove a general perturbation result (Lemma~\ref{lem-zero-perturbation}) in Section~\ref{sec-zero-perturbation}, where the size of perturbation is related to the intrinsic distance on $\mathcal C^{\Lambda_N}$.
\item In Section~\ref{sec-zero-crossing-dimension}, we apply Lemma~\ref{lem-zero-perturbation} by  bounding  $d_{\mathcal C^{\Lambda_N}}$ from below by the $\ell_1$-distance and correspondingly setting the perturbation amount to $1/N$, thereby proving  Lemma~\ref{lem-zero-bound-hard-crossing}. As a consequence, we verify \eqref{eq-zero-crossing-prob}.
\item In Section~\ref{sec-zero-main-thm}, we apply Lemma~\ref{lem-zero-perturbation} again by applying a lower bound on  $d_{\mathcal C^{\Lambda_N}}$ from Proposition~\ref{prop-zero-crossing-dimension}. This allows us to derive Lemma~\ref{lem-zero-m-star}. As a consequence, we prove in Lemma~\ref{lem-zero-m-N-bound}  polynomial decay for $m_N$ with large power, which is then  enhanced to exponential decay by a standard argument.
\end{itemize}

\subsection{A perturbative analysis}\label{sec-zero-perturbation}

We first introduce some notation. For $A\subseteq \mathbb Z^2$, we set $h_A = \sum_{v\in A} h_v$. For $A, B \subset \mathbb Z^2$, we denote by $E(A, B) = \{\langle u, v\rangle: u\sim v, u\in A, v\in B\}$. Note that we treat $\langle u, v\rangle$ as an ordered edge. For simplicity, we will only consider $N = 2^n$ for $n\geq 10$. Let $\mathcal A_N = \Lambda_N \setminus \Lambda_{N/2}$ be an annulus.   Define $\{\tilde h^{(N)}_v: v\in \Lambda_N\}$ to be a perturbation of the original field parameterized by $\Delta>0$, as follows:
\begin{equation}\label{eq-zero-def-tilde-h}
\tilde h^{(N)}_v  =
h_v + \Delta \mbox{ for } v\in \Lambda_N\,.
\end{equation}
We will use  $\tilde H^{\Lambda_N, \pm}(\sigma)$, $\tilde \sigma^{\Lambda_N, \pm}$, $\tilde \xi^{\Lambda_N}$, $\tilde {\mathcal C}^{\Lambda_N}$ to denote the corresponding tilde versions of $H^{\Lambda_N, \pm}(\sigma)$, $\sigma^{\Lambda_N, \pm}$, $\xi^{\Lambda_N}$, ${\mathcal C}^{\Lambda_N}$ respectively, i.e., defined analogously but with respect to the field $\{\tilde h^{(N)}_v\}$. In addition, define $\mathcal C_*^{\Lambda_N} = \tilde {\mathcal C}^{\Lambda_N} \cap  {\mathcal C}^{\Lambda_N}$ (so $\mathcal C_*^{\Lambda_N}$ is the intersection of disagreements with respect to the original and the perturbed field; in informal discussions we will refer to vertices in $\mathcal C_*^{\Lambda_N}$ as disagreements too).
\begin{lemma}\label{lem-zero-perturbation}
Consider $K, \Delta > 0$. Define $\{\tilde h^{(N)}_v: v\in \Lambda_N\}$
as in \eqref{eq-zero-def-tilde-h}. The following two conditions cannot hold simultaneously:

(a) $d_{\mathcal C_*^{\Lambda_N}}(\partial \Lambda_{N/4}, \partial \Lambda_{N/2}) \geq K$;

(b) $|\mathcal C_*^{\Lambda_N} \cap \Lambda_{N/4}| \cdot \Delta> \frac{8}{K} |\mathcal C_*^{\Lambda_N} \cap \mathcal A_{N/2}|$.
\end{lemma}
\begin{proof}
Suppose otherwise both (a) and (b) hold. Let  $B_k = \{v\in \mathcal A_{N/2}: d_{\mathcal C_*^{\Lambda_N}}(\partial \Lambda_{N/4}, v) = k\}$, for $k=1, \ldots, K$. Note that $B_k \subset \mathcal C_*^{\Lambda_N} \cap \mathcal A_{N/2}$ for all $1\leq k\leq K$ by (a). It is obvious that the $B_k$'s are disjoint from each other, and thus there exists a minimal value $k_*$ such that
\begin{equation}\label{eq-zero-B-*}
|B_{k_*}| \leq K^{-1}| \mathcal C_*^{\Lambda_N} \cap \mathcal A_{N/2}|\,.
\end{equation}
 Let
$$S =(\mathcal C_*^{\Lambda_N} \cap \Lambda_{N/4} ) \cup \cup_{k=1}^{k^*-1} B_k\,,$$
and for $\tau \in \{\minus, \zero, \plus\}$, define
\begin{equation}\label{eq-zero-definition-g}
g(S, \tau) = \{\langle u, v\rangle \in E(S, S^c): \xi^{\Lambda_N}_v = \tau \} \mbox{ and } \tilde g(S, \tau) = \{\langle u, v\rangle \in E(S, S^c): \tilde \xi^{\Lambda_N}_v = \tau \}\,.
\end{equation}
Note that for any $v\in \Lambda_N$ with $\xi_v^{\Lambda_N} = \zero$ we have $\sigma^{\Lambda_N, +}_v = 1$. Since $\xi^{\Lambda_N}_{v} = \zero$ for $v\in S$ (which implies that $\sigma^{\Lambda_N, +}_v = 1$ for $v\in S$),
\begin{equation}\label{eq-zero-S-condition}
h_S + |g(S, \plus)| - |g(S, \minus)| + |g(S, \zero)| \geq 0\,,
\end{equation}
because if \eqref{eq-zero-S-condition} does not hold, then $H^{\Lambda_N, +}(\sigma') < H^{\Lambda_N, +}(\sigma^{\Lambda_N, +})$ where $\sigma'$ is obtained from $\sigma^{\Lambda_N, +}$ by flipping its value on $S$, thus contradicting the minimality of $H^{\Lambda_N, +}(\sigma^{ \Lambda_N, +})$. In addition, by monotonicity (with respect to the external field), we have $g(S, \zero) \subset \tilde g(S, \zero)\cup \tilde g(S, \plus)$, $g(S, \plus) \subset \tilde g (S, \plus)$, and thus
$$|\tilde g(S, \plus)| - |g(S, \plus)| \geq | g(S, \zero) \setminus \tilde g(S, \zero)|\,.$$
Similarly, we have $\tilde g(S, \minus) \subset g(S, \minus)$ and $\tilde g(S, \zero) \subset g(S, \minus) \cup g(S, \zero)$, and thus
$$|g(S, \minus)| - |\tilde g(S, \minus)| \geq |\tilde g(S, \zero) \setminus g(S, \zero)|\,.$$
 By our definition of $B_k$'s, we see that $\tilde g(S, \zero) \cap g(S, \zero) = E(S, B_{k_*})$. Therefore, \eqref{eq-zero-S-condition} and the preceding two displays imply that
\begin{align*}
\tilde h^{(N)}_S + |\tilde g(S, \plus)| - |\tilde g(S, \minus)| - |\tilde g(S, \zero)| &\geq \tilde h^{(N)}_S + | g(S, \plus)| - | g(S, \minus)| + |g(S, \zero)| - 2| E(S, B_{k_*})|\\
&\geq |S| \Delta - 8|B_{k_*}| >0\,,
\end{align*}
where the last inequality follows from (b) and \eqref{eq-zero-B-*}. The preceding inequality implies  $\tilde H^{\Lambda_N, -}(\sigma') < \tilde H^{\Lambda_N, -}(\tilde \sigma^{\Lambda_N, -})$ where $\sigma'$ is obtained from $\tilde \sigma^{\Lambda_N, -}$ by flipping its value on $S$. This contradicts the minimality of $\tilde H^{\Lambda_N, -}(\tilde \sigma^{\Lambda_N, -})$, completing the proof of the lemma.
\end{proof}

\begin{lemma}\label{lem-zero-percolation-property}
For any $x_v \geq 0$ for  $v\in \Lambda_N$, let $\check h^{(N)}_v = h_v + x_v$ for $v\in \Lambda_N$ (we will use  $\check H^{\Lambda_N, \pm}(\sigma)$, $\check \sigma^{\Lambda_N, \pm}$, $\check \xi^{\Lambda_N}$, $\check {\mathcal C}^{\Lambda_N}$ to denote the corresponding $\check{}$ versions of $H^{\Lambda_N, \pm}(\sigma)$, $\sigma^{\Lambda_N, \pm}$, $\xi^{\Lambda_N}$, ${\mathcal C}^{\Lambda_N}$).
 Then with probability 1, for any $v\in \mathcal C^{\Lambda_N} \cap \check {\mathcal C}^{\Lambda_N}$ there is a path in $ \mathcal C^{\Lambda_N} \cap \check {\mathcal C}^{\Lambda_N}$ joining $v$ and $\partial \Lambda_N$.
\end{lemma}
\begin{proof}
The proof is similar to that of Lemma~\ref{lem-zero-perturbation}, and in a way it is the case of $K = \infty$ there.

Suppose that the claim is not true. Then take $v\in  \mathcal C^{\Lambda_N} \cap \check {\mathcal C}^{\Lambda_N}$ (for which the claim fails), and let $S$ be the connected component in $\mathcal C^{\Lambda_N} \cap \check {\mathcal C}^{\Lambda_N}$ that contains $v$ (thus $S$ is not neighboring $\partial \Lambda_N$). Define $g(S, \tau)$ as in \eqref{eq-zero-definition-g} and define $\check g(S, \tau) = \{\langle u, v\rangle \in E(S, S^c): \check \xi^{\Lambda_N}_v = \tau \}$ . Similar to \eqref{eq-zero-S-condition}, we have that
$$h_S + |g(S, \plus)| - |g(S, \minus)| + |g(S, \zero)| \geq 0\,.$$
In our case, $g(S, \zero) \cup g(S, \plus) \subset  \check g(S, \plus)$ and $\check g(S, \zero) \cup \check g(S, \minus) \subset g(S, \minus)$. Therefore,
$$\check h^{(N)}_S + |\check g(S, \plus)| - |\check g(S, \minus)| - |\check g(S, \zero)| \geq  h_S + | g(S, \plus)| - | g(S, \minus)| + |g(S, \zero)| \geq 0\,.$$
The preceding inequality implies that $\check H^{\Lambda_N, -}(\sigma') \leq \check H^{\Lambda_N, -}(\check \sigma^{ \Lambda_N,-})$ where $\sigma'$ is obtained from $\check \sigma^{ \Lambda_N, -}$ by flipping its value on $S$. This happens with probability 0 since the ground state is unique with probability 1.
\end{proof}

\subsection{Proof of Proposition~\ref{prop-zero-crossing-dimension}}\label{sec-zero-crossing-dimension}

In this section, we will set $K = K(N)= N/4$, and $\Delta = \Delta(N)= \gamma/N$ for an absolute constant $\gamma>0$ to be selected, and we consider $\tilde h^{(N)}$ as in \eqref{eq-zero-def-tilde-h}.
In this case Condition (a) in Lemma~\ref{lem-zero-perturbation} holds trivially. For convenience, we use $ \P_N$ to denote the probability measure with respect to the field $\{ h_v: v\in \Lambda_N\}$ and use $\tilde \P_N$ to denote the probability measure with respect to $\{\tilde h^{(N)}_v: v\in \Lambda_N\}$.
\begin{lemma}\label{lem-zero-perturbation-contiguous}
Recall that $\epsilon$ is the variance parameter for the field $\{h_v\}$. For any $p>0$, there exists $c = c(\epsilon,p, \gamma)>0$ such that for any event $E_N$ with $\tilde \P_N(E_N)\geq p$, we have that
$$\P_N(E_N) \geq c\,.$$
\end{lemma}
\begin{proof}
There exists a constant $C>0$ such that $\tilde \P_N(|\tilde h^{(N)}_{\Lambda_{N}} - \Delta |\Lambda_N|| \geq C \epsilon N) \leq p/2$. Thus we have
\begin{equation}\label{eq-zero-prob-E-N-bound}
\tilde \P_N(E_N; |\tilde h^{(N)}_{\Lambda_{N}} - \Delta |\Lambda_N|| \leq C \epsilon N) \geq p/2\,.
\end{equation}
Also, by a straightforward Gaussian computation, we see that
\begin{equation}\label{eq-zero-change-of-measure}
\frac{d  \P_N}{d \tilde \P_N}  = \exp\big\{ - \frac{\Delta(\tilde h^{(N)}_{\Lambda_{N}} - \Delta |\Lambda_N|)}{\epsilon^2} \big\} \exp\big\{\frac{-\Delta^2 |\Lambda_N|}{2\epsilon^2}\big\}
\end{equation}
and thus there exists $\iota = \iota (\epsilon)>0$ such that
$$\frac{d  \P_N}{d \tilde \P_N} \geq \iota \mbox{ provided that } |\tilde h^{(N)}_{\Lambda_{N}} - \Delta |\Lambda_N|| \leq C \epsilon N\,.$$
Combined with \eqref{eq-zero-prob-E-N-bound}, this completes the proof of the lemma.
\end{proof}
For any annulus $\mathcal A$, we denote by $\mathrm{Cross}_{\mathrm{hard}}(\mathcal A, \mathcal C)$ the event that there is a contour in $\mathcal C$ which separates the inner and outer boundaries of $\mathcal A$, and by $\mathrm{Cross}_{\mathrm{easy}}(\mathcal A, \mathcal C)$ the event that there is a path in $\mathcal C$ which connects the inner and outer boundaries of $\mathcal A$.
\begin{lemma}\label{lem-zero-bound-hard-crossing}
There exists $\delta = \delta(\epsilon)>0$ such that
$$
\min\{\P(\mathrm{Cross}_{\mathrm{hard}}(\Lambda_{N/8}\setminus \Lambda_{N/32}, \mathcal C^{\Lambda_N})), \P(\mathrm{Cross}_{\mathrm{easy}}(\Lambda_{N/8}\setminus \Lambda_{N/32}, \mathcal C^{\Lambda_N})) \}\leq 1- \delta \mbox{ for all } N \geq 32.$$
\end{lemma}
\begin{proof}
We first provide a brief discussion on the outline of the proof. We refer to the disagreements on $\Lambda_{N/32}$ with plus/minus boundary conditions posed on $\partial \Lambda_{N/8}$ as the ``enhanced'' disagreements (the word enhanced is chosen since the enhanced disagreements stochastically dominate the disagreements with boundary conditions on $\partial \Lambda_N$ by monotonicity of the Ising model). Note that the set of disagreements in $\mathcal A_{N/2}$ is stochastically dominated by the union of a constant number of copies of enhanced disagreements, which are independent of the enhanced disagreements in $\Lambda_{N/32}$. Therefore, with positive probability the number of enhanced disagreements in $\Lambda_{N/32}$ is larger than (up to a constant factor) the number of disagreements in $\mathcal A_{N/2}$ (see \eqref{eq-zero-prob-E}). On this event, (modulo a caveat) by Lemma~\ref{lem-zero-perturbation}  at least one of the enhanced disagreements is not a disagreement when considering boundary conditions on $\partial \Lambda_N$ --- this yields the desired statement as incorporated in {\bf Case 1} below. In {\bf Case 2}, we tighten the argument by addressing the caveat which is the scenario that the enhanced disagreement is empty (this is relatively simple).

We are now ready to carry out the formal proof. We can write $\mathcal A_{N/2} = \cup_{i=1}^r A_i$ where each $A_i$ is a  box of side length $N/16$ (so a copy of $\Lambda_{N/32}$) and $r\geq 16$ is a fixed integer. For a box $A$, denoting by $A^{\mathrm{Big}}$ as the concentric box of $A$ whose side length is $4 \ell_A$. We have that (see Figure~\ref{figure-2.7})
\begin{equation}\label{eq-zero-disjointness}
A_i^{\mathrm{Big}} \cap \Lambda_{N/8} = \emptyset  \mbox{ and } A_i^{\mathrm{Big}}  \subset \Lambda_N \mbox{ for all } 1\leq i\leq r.
\end{equation}
For any $A\subset \Lambda_N$, let $\bar {\mathcal C}^A$ be defined as $\mathcal C^A$ but replacing $\{h_v: v\in A\}$ by $\{\tilde h^{(N)}_v: v\in A\}$ (note that $\bar {\mathcal C}^{\Lambda_{N/2}}$ is different from $\tilde {\mathcal C}^{\Lambda_{N/2}}$, which is defined with respect to $\tilde h^{(N/2)}$). Write $\mathcal C_\diamond^{A} = \mathcal C^A \cap \bar {\mathcal C}^A$.
Write $X_i = |\mathcal C_\diamond^{A_i^{\mathrm{Big}}} \cap A_i|$ and $X = |\mathcal C_\diamond^{\Lambda_{N/8}} \cap \Lambda_{N/32}|$. Clearly, $X_i$'s and $X$ are identically distributed and by \eqref{eq-zero-disjointness} $X_i$'s are independent of $X$ (but $X_i$'s are not mutually independent). Let $\theta = \inf\{x: \P(X \leq x) \geq 1 - 1/2r \}$. Thus,
\begin{equation}\label{eq-zero-prob-E}
\mathbb P(X \geq \max_{1\leq i\leq r} X_i, X\geq \theta) \geq \P(X\geq \theta) \P(\max_{1\leq i\leq r} X_i \leq \theta) \geq 1/4r\,.
\end{equation}
The rest of the proof divides into two cases.
\begin{center}
\begin{figure}[h]
  \includegraphics[width=15cm]{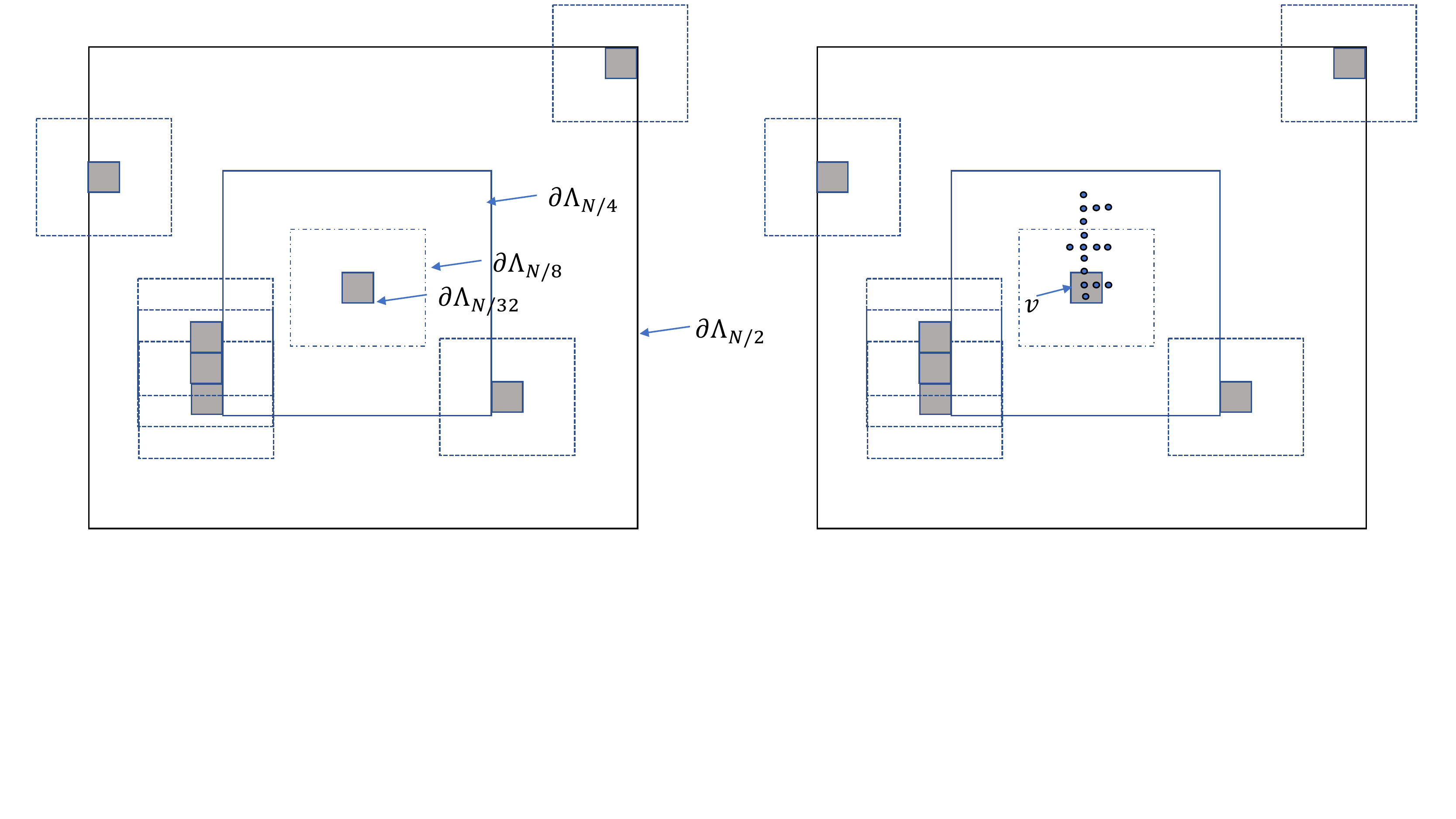}\\ \vspace{-3cm}
 \caption{Illustration for the geometric setup of the proof for Lemma 2.7. In the picture on the left we cover $\mathcal A_{N/2}$ by a collection of translated copies of $\Lambda_{N/32}$ (the grey boxes) --- we only draw out a few copies for an illustration. Note that the (4-times) enlargements of translated copies (while overlapping among themselves) are all disjoint with $\Lambda_{N/8}$. The picture on the right illustrates the scenario in Case 1: for some $v \in \mathcal C_\diamond^{\Lambda_{N/8}} \setminus \mathcal C^{\Lambda_{N}}$, we draw its component with the same $\xi^{\Lambda_N}$-value and this component necessarily goes out of $\Lambda_{N/8}$.}\label{figure-2.7}
\end{figure}
\end{center}

\noindent {\bf Case 1:} $\theta>0$.
Let $\mathcal E = \{  |\mathcal C_\diamond^{\Lambda_{N/8}} \cap \Lambda_{N/32}| \geq r^{-1} | \mathcal C_*^{\Lambda_N} \cap \mathcal A_{N/2}|\} \cap\{ |\mathcal C_\diamond^{\Lambda_{N/8}} \cap \Lambda_{N/32}|  >0\}$. By \eqref{eq-zero-monotonicity} and \eqref{eq-zero-disjointness}, we have $|\mathcal C_*^{\Lambda_N} \cap \mathcal A_{N/2}| \leq \sum_{i=1}^r X_i$. Combined with \eqref{eq-zero-prob-E}, it gives that $\P(\mathcal E) \geq 1/4r$. Setting $\gamma = 100 r$, we get that $|\mathcal C_\diamond^{\Lambda_{N/8}} \cap \Lambda_{N/32}| \cdot \Delta> 16 K^{-1}  | \mathcal C_*^{\Lambda_N} \cap \mathcal A_{N/2}|$ on $\mathcal E$. By Lemma~\ref{lem-zero-perturbation}, on $\mathcal E$ there is at least one vertex $v\in \mathcal C_\diamond^{\Lambda_{N/8}} \cap \Lambda_{N/32}$ but $v\not\in \mathcal C_*^{\Lambda_N}$. So either $v \not\in \mathcal C^{\Lambda_N} $ or $v \not\in \tilde {\mathcal C}^{\Lambda_N} $ on $\mathcal E$. Assume that $v \not\in \mathcal C^{\Lambda_N}$ and the other case can be treated similarly.

We will use the following property: for any connected set $\mathcal A$, $u\not\in \mathcal C^{\mathcal A}$ if and only if there exists a connected set $A\subset \mathcal A$ with $u\in A$ such that $\xi^A_w= \plus$ for all $w\in A$ or $\xi^A_w= \minus$ for all $w\in A$. The ``if'' direction of the property follows from \eqref{eq-zero-monotonicity}. For the ``only if'' direction, we assume without loss that $\xi_u^{\mathcal A} = \plus$ and let $A$ be the connected component containing $u$ where the $\xi^{\mathcal A}$-value is $\plus$. Note $\sigma^{\mathcal A, -}_w = -1$ for all $w\in \partial A$ and $\sigma^{\mathcal A, -}_w = 1$ for all $w\in A$. This implies that $\xi^A_w = \plus$ for all $w\in A$.

By the preceding property, there exists a connected set $A\subset \Lambda_N$ with $v\in A$ such that $\xi^{A}_w = \plus$ for all $w\in A$ or $\xi^{A}_w = \minus$ for all $w\in A$ (see Figure~\ref{figure-2.7} for an illustration). In addition, $A$ cannot be contained in $\Lambda_{N/8}$ since otherwise it contradicts $v\in \mathcal C^{\Lambda_{N/8}}$. By planar duality, this implies that on $\mathcal E$, either $\mathrm{Cross}_{\mathrm{hard}}(\Lambda_{N/8}\setminus \Lambda_{N/32}, \mathcal C^{\Lambda_N})$ or $\mathrm{Cross}_{\mathrm{hard}}(\Lambda_{N/8}\setminus \Lambda_{N/32}, \tilde {\mathcal C}^{\Lambda_N})$ does not occur (the second case corresponds to the case when $v\not\in \tilde {\mathcal C}^{\Lambda_N}$). Therefore,
$$\P((\mathrm{Cross}_{\mathrm{hard}}(\Lambda_{N/8}\setminus \Lambda_{N/32}, \mathcal C^{\Lambda_N}))^c) + \P((\mathrm{Cross}_{\mathrm{hard}}(\Lambda_{N/8}\setminus \Lambda_{N/32}, \tilde {\mathcal C}^{\Lambda_N}))^c) \geq \P(\mathcal E) \geq 1/4r\,.$$
Combined with Lemma~\ref{lem-zero-perturbation-contiguous}, this completes the proof of the lemma.

\noindent {\bf Case 2:}  $\theta = 0$. Applying a simple union bound (by using 16 copies of $\Lambda_{N/32}$ to cover $\Lambda_{N/8}$, and a derivation similar  to $|\mathcal C_*^{\Lambda_N} \cap \mathcal A_{N/2}| \leq \sum_{i=1}^r X_i$) we get that $\P(\mathcal C_*^{\Lambda_N} \cap  \Lambda_{N/8} = \emptyset) \geq 1/2$. We assume without loss that $ \P(\mathrm{Cross}_{\mathrm{easy}}(\Lambda_{N/8}\setminus \Lambda_{N/32}, \mathcal C^{\Lambda_N})) \geq 3/4$ (otherwise there is nothing further to prove), and thus
$$ \P(\mathrm{Cross}_{\mathrm{easy}}(\Lambda_{N/8}\setminus \Lambda_{N/32}, \mathcal C^{\Lambda_N})\mbox{ and }\mathcal C_*^{\Lambda_N} \cap  \Lambda_{N/8} = \emptyset ) \geq 1/4\,.$$
On the event $\mathrm{Cross}_{\mathrm{easy}}(\Lambda_{N/8}\setminus \Lambda_{N/32}, \mathcal C^{\Lambda_N})\mbox{ and }\mathcal C_*^{\Lambda_N} \cap  \Lambda_{N/8} = \emptyset$, the easy crossing (joining two boundaries of $\Lambda_{N/8}\setminus \Lambda_{N/32}$) in $\mathcal C^{\Lambda_N}$ becomes an easy crossing with $\tilde \xi^{\Lambda_N}$-values $\plus$. Thus, by planar duality, it prevents existence of a contour surrounding $\Lambda_{N/32}$ in $(\Lambda_{N/8}\setminus \Lambda_{N/32}) \cap \tilde {\mathcal C}^{\Lambda_N}$. Therefore,
$$\P((\mathrm{Cross}_{\mathrm{hard}}(\Lambda_{N/8}\setminus \Lambda_{N/32}, \tilde {\mathcal C}^{\Lambda_N}))^c) \geq 1/4\,.$$
Combined with Lemma~\ref{lem-zero-perturbation-contiguous}, this completes the proof of the lemma.
\end{proof}

\begin{proof}[Proof of \eqref{eq-zero-crossing-prob}]
Let $N = \min\{2^n: 2^{n+2} \geq \ell_A\}$. By our assumption on $A$, it is clear that we can position four copies $A_1, A_2, A_3, A_4$ of $A$ by translation or rotation by 90 degrees so that (see the left of Figure~\ref{figure})
\begin{itemize}
\item $A_1, A_2, A_3, A_4 \subset \Lambda_{N/8} \setminus \Lambda_{N/32}$.
\item  The union of any crossings through $A_1, A_2, A_3, A_4$ in their longer directions surrounds $\Lambda_{N/32}$.
\item $\Lambda_N \subset A^{\mathrm{Large}}_i$ for $1\leq i\leq 4$.
 \end{itemize}
Set $p = \P(\mathrm{Cross}(A, \mathcal C^{A^{\mathrm{Large}}}))$ (note that $p$ depends on the dimension of $A$ and also the orientation of $A$). By rotation symmetry and \eqref{eq-zero-monotonicity} we see that $\P(\mathrm{Cross}(A_i, \mathcal C^{\Lambda_N}) ) \geq \P(\mathrm{Cross}(A_i, \mathcal C^{A_i^{\mathrm{Large}}}) ) =   p$.  In what follows, we denote $\mathcal A =  \Lambda_{N/8} \setminus \Lambda_{N/32}$. Then, by  $\P(\mathrm{Cross}(A_i, \mathcal C^{\Lambda_N}) ) \geq p$ and a simple union bound, we get that
 \begin{equation}\label{eq-zero-cross-A-hard}
 \P(\mathrm{Cross}_{\mathrm{hard}}(\mathcal A, \mathcal C^{\Lambda_N})) \geq \P (\cap_{i=1}^4 \mathrm{Cross}(A_i, \mathcal C^{\Lambda_N}))\geq 1 - 4(1-p)\,.
 \end{equation}
 Similarly, we can arrange two copies  $A_a, A_b$ of $A$ obtained by translation and rotation by 90 degrees such that $\Lambda_N \subset A_a^{\mathrm{Large}}, A_b^{\mathrm{Large}}$ and that the union of any two crossings through $A_a^{\mathrm{Large}}, A_b^{\mathrm{Large}}$ in the longer direction connects the two boundaries of $\mathcal A$ (see the right of Figure~\ref{figure}). This implies that
 \begin{equation}\label{eq-zero-cross-A-easy}
 \P(\mathrm{Cross}_{\mathrm{easy}}(\mathcal A, \mathcal C^{\Lambda_N})) \geq  \P (\mathrm{Cross}(A_a, \mathcal C^{\Lambda_N}) \cap \mathrm{Cross}(A_b, \mathcal C^{\Lambda_N}))\geq 1 - 2(1-p)\,.
 \end{equation}
 Combined with \eqref{eq-zero-cross-A-hard} and Lemma~\ref{lem-zero-bound-hard-crossing}, it yields that $p\leq 1 - \delta$ for some $\delta = \delta(\epsilon)>0$ as required.
\end{proof}

The following standard lemma will be applied several  times below. Before presenting the lemma, we first provide a definition.
\begin{definition}\label{def-percolation-process}
 Divide $\Lambda_N$ into disjoint boxes of side lengths $N' \leq N$ where $N' = 2^{n'}$ for some $n'\geq 1$, and denote by $\mathcal B(N, N')$ the collection of such boxes. Consider a percolation process on $\mathcal B(N, N')$, where each box $B\in \mathcal B(N, N')$ is regarded open or closed randomly. For $C, p>0$, we say that the percolation process satisfies the $(N, N', C, p)$-condition if for each $B\in \mathcal B(N, N')$, there exists an event $E_B$ such that
\begin{itemize}
\item On $E_B^c$, $B$ is closed.
\item $\P(E_B) \leq p$ for each $B$.
\item If $\min_{x\in B_i, y\in B_j}|x-y|_\infty \geq CN'$ for all $1\leq i<j\leq k$, then the events $E_{B_1}, \ldots, E_{B_k}$ are  mutually independent.
\end{itemize}
Furthermore, we say two boxes $B_1, B_2$ are adjacent if $\min_{x_1\in B_1, x_2\in B_2} |x_1 - x_2|_\infty  \leq 1$, and we say a collection of boxes is a lattice animal if these boxes form a connected graph.
\end{definition}
\begin{lemma}\label{lem-zero-enhance}
For any $C>0$, there exists $p>0$ such that for all $N$ and $N'\leq N$ and any percolation process on $\mathcal B(N, N', C, p)$ satisfying the $(N,N', C, p)$-condition, we have
$$\P(\mbox{there exists a lattice animal of open boxes on } \mathcal B(N, N') \mbox{ of size  at least } k) \leq (\tfrac{N}{N'})^2 2^{-k}\,.$$
\end{lemma}
\begin{proof}
 On the one hand, the number of lattice animals of size exactly $k$ is bounded by  $(\tfrac{N}{N'})^2 8^{2k}$ (the bound comes from first choosing a starting box, and then encoding the lattice animal by a  surrounding contour on $ \mathcal B(N, N')$ of length $2k$). On the other hand, for any $k$ such boxes, we can extract a sub-collection of $ck$ boxes (here $c>0$ is a constant that depends only on $C$) such that the pairwise distances of boxes in this sub-collection are at least $CN'$; hence the probability that all these $k$ boxes are open is at most $p^{ck}$. The proof of the lemma is then completed by a simple union bound, employing the $(N, N', C, p)$-condition.
\end{proof}

 \begin{figure}[h]
  \includegraphics[width=15cm]{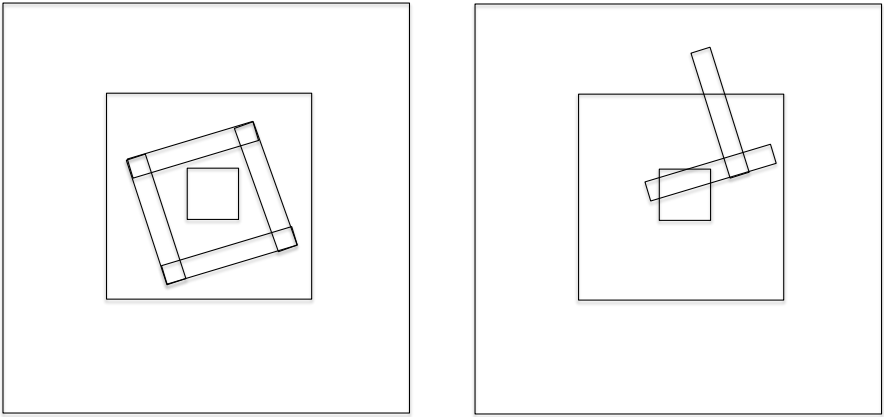}
\\ \vspace{-5.5cm} \caption{On both left and right, the three concentric square boxes are $\Lambda_N$, $\Lambda_{N/8}$ and $\Lambda_{N/32}$ respectively. On the left, the four rectangles are $A_1, A_2, A_3, A_4$ and on the right the two rectangles are $A_a, A_b$.}\label{figure}
\end{figure}

\begin{proof}[Proof of Proposition~\ref{prop-zero-crossing-dimension}]
Let $N'= N^{1 - (\frac{\alpha - 1}{10} \wedge \frac{1}{10})}$, where $\alpha$ is as in \eqref{eq-zero-box-counting-weak}. For each $B\in \mathcal B(N, N')$, we say $B$ is open if $d_{\mathcal C^{B^{\mathrm{Large}}}} (\partial B, \partial B^{\mathrm{large}}) \leq (N')^{\alpha}$, where $B^{\mathrm{large}}$ is the box  concentric with $B$ of doubled side length and $B^{\mathrm{Large}}$ (as we recall) is a concentric box of $B$ with side length $32\ell_B$. By \eqref{eq-zero-box-counting-weak}, we see that  this percolation process satisfies the $(N, N', 64,p)$-condition where $p \to 0$ as $N\to \infty$. Now,  in order that $ d_{\mathcal C^{\Lambda_N}} (\partial \Lambda_{N/4}, \partial \Lambda_{N/2}) \leq (N')^{\alpha}$, there must exist an open lattice animal on  $\mathcal B(N, N')$ of size at least $\frac{N}{16N'}$. Applying Lemma~\ref{lem-zero-enhance} completes the proof of Proposition~\ref{prop-zero-crossing-dimension} (since $(\alpha(1 - (\frac{\alpha - 1}{10} \wedge \frac{1}{10}))>1$).
\end{proof}

\subsection{Proof of Theorem~\ref{thm-main-zero}}\label{sec-zero-main-thm}

In this subsection, we will show that the probability for $\{o\in \mathcal C^{\Lambda_N}\}$ has a polynomial decay with large power (Lemma~\ref{lem-zero-m-N-bound}), which then yields Theorem~\ref{thm-main-zero} by a standard application of Lemma~\ref{lem-zero-enhance}. In order to prove Lemma~\ref{lem-zero-m-N-bound}, we first provide a bound on the probability for $\{o\in \mathcal C^{\Lambda_N}_*\}$ (Lemma~\ref{lem-zero-m-star}), whose proof crucially relies on Proposition~\ref{prop-zero-crossing-dimension}.

Let $\alpha>1$ be as in Proposition~\ref{prop-zero-crossing-dimension} (note that we can assume without loss that $\alpha \leq 2$). Let $\sqrt{1/\alpha} < \alpha' < 1$ (and thus we have $\alpha(\alpha')^2 > 1$).
\begin{lemma}\label{lem-zero-m-star}
For $N^\diamond \geq 16$, set $\Delta = (N^\diamond)^{-\alpha(\alpha')^2}$ and let $\tilde h^{(N)}$ be defined as in \eqref{eq-zero-def-tilde-h} for $N\leq N^\diamond$. Write $m^\diamond_N = m^\diamond_N(N^\diamond) =  \P(o \in \mathcal C_*^{\Lambda_N})$. Then there exists $C= C(\epsilon)>0$ such that  $m^\diamond_{N^\diamond} \leq C (N^\diamond)^{-6}$.
\end{lemma}
\begin{remark}
(1) In this lemma, regardless of the size of the box under consideration, the amount of perturbation $\Delta$ in our field $\tilde h^{(N)}$ only depends on $N^\diamond$. This is crucial for  \eqref{eq-zero-monotonicity-consequence} below.
(2) Since $\alpha(\alpha')^2 > 1$, we have that $\Delta \ll 1/N^\diamond$ (this is crucial for getting a large power in the polynomial bound as in Lemma~\ref{lem-zero-m-N-bound}). (3) Since our perturbation  $\Delta = (N^\diamond)^{-\alpha(\alpha')^2}$ applies to all $N\leq N^\diamond$, when $N$ is very small in comparison of $N^\diamond$ the perturbation is possibly too mild and thus we may not have a good control on $\mathcal C^{\Lambda_N}_*$. However, this is not a problem because in the proof below we will only consider $N \geq (N^\diamond)^{\alpha'}$ (for which the perturbation is still significant).
\end{remark}
\begin{proof}
Write $K = (N^\diamond)^{\alpha \alpha'}$.
We claim it suffices to show that there exists $N_0 = N_0(\epsilon)$ such that  for  $N^\diamond \geq N_0$
\begin{equation}\label{eq-zero-m-recursion-star}
m^\diamond_{2N} \leq K^{-\frac{1-\alpha'}{2}} m^\diamond_{N/2} \mbox{ for }  (N^\diamond)^{\alpha'} \leq  N \leq N^\diamond\,.
\end{equation}
Indeed, since $K = (N^\diamond)^{\alpha \alpha'}$ we have $K^{-\frac{1-\alpha'}{2}} \ll 1$ for large $N^\diamond$, and thus \eqref{eq-zero-m-recursion-star} yields the desired bound on $m^\diamond(N^\diamond)$ by recursion.

We now turn to the proof of \eqref{eq-zero-m-recursion-star}.
Suppose that \eqref{eq-zero-m-recursion-star} fails for some  $(N^\diamond)^{\alpha'} \leq  N \leq N^\diamond$. Since $\Lambda_N \subset v + \Lambda_{2N}$ for all $v\in \Lambda_{N/4}$ and $v + \Lambda_{N/2} \subset \Lambda_N$ for all $v\in \mathcal A_{N/2}$,  by \eqref{eq-zero-monotonicity} we see
\begin{equation}\label{eq-zero-monotonicity-consequence}
\E|\mathcal C_*^{\Lambda_N} \cap \Lambda_{N/4}|  \geq \frac{N^2}{32} m^\diamond_{2N} \mbox{ and } \E |\mathcal C_*^{\Lambda_N} \cap \mathcal A_{N/2}| \leq N^2 m^\diamond_{N/2}\,.
\end{equation}
Together with the assumption that \eqref{eq-zero-m-recursion-star} fails, this yields that
$$ \E|\mathcal C_*^{\Lambda_N} \cap \Lambda_{N/4}|  > 32^{-1} K^{-\frac{1-\alpha'}{2}}\E |\mathcal C_*^{\Lambda_N} \cap \mathcal A_{N/2}|\,.$$
Since $|\mathcal C_*^{\Lambda_N} \cap \Lambda_{N/4}| $ and $|\mathcal C_*^{\Lambda_N} \cap \mathcal A_{N/2}|$ are integer-valued and are at most $N^2$, the preceding inequality implies that (recall that $\alpha' > 1/\sqrt{\alpha} \geq 1/\sqrt{2}$)
$$ \P(|\mathcal C_*^{\Lambda_N} \cap \Lambda_{N/4}|  > 64^{-1} K^{-\frac{1-\alpha'}{2}} |\mathcal C_*^{\Lambda_N} \cap \mathcal A_{N/2}|) \geq \frac{1}{32 N^3}\,.$$
Now, set $N_0 = N_0(\epsilon)$ sufficiently large so that
\begin{equation}\label{eq-zero-N-0}
\frac{1}{10^6 N^3} > \kappa^{-1} e^{- N^{\kappa}} \mbox{and } 64^{-1} K^{-\frac{1-\alpha'}{2}} > \frac{8}{K\Delta} \mbox{ for all } N\geq (N_0)^{\alpha'}\,.
\end{equation}
Therefore, by Proposition~\ref{prop-zero-crossing-dimension}, there is a positive probability such that
$$|\mathcal C_*^{\Lambda_N} \cap \Lambda_{N/4}|  > 64^{-1} K^{-\frac{1-\alpha'}{2}} |\mathcal C_*^{\Lambda_N} \cap \mathcal A_{N/2}|\mbox{ and } d_{\mathcal C_*^{\Lambda_N}}(\partial \Lambda_{N/4}, \partial \Lambda_{N/2}) \geq K\,.$$
In particular, there exists at least one instance for the two events in the preceding display to occur simultaneously.
This contradicts Lemma~\ref{lem-zero-perturbation}, thus completing the proof of the lemma.
\end{proof}

In the proof of Lemma~\ref{lem-zero-m-N-bound} below, it is important for us to have independence between different scales. To this end, it is useful to consider a perturbation which only occurs in an annulus. In order to make a difference in notation from the previous perturbation (which occurs in a whole box), for $\Delta(N)>0$ we define (we emphasize the dependence of $\Delta$ on $N$ in the notation here since later in Lemma~\ref{lem-zero-m-N-bound} we will consider perturbations for different $N$'s simultaneously)
\begin{equation}\label{eq-zero-def-tilde-h-annulus}
\hat h^{(N)}_v =
\begin{cases}
h_v + \Delta(N) &\mbox{ for } v\in \Lambda_{N} \setminus \Lambda_{N/4}\,, \\
 h_v &\mbox{ for }v \in \Lambda_{N/4}\,.
\end{cases}
\end{equation}
We then define $\hat {\mathcal C}^{\Lambda_N}$ similar to $\mathcal C^{\Lambda_N}$ but with respect to the field $\{\hat h^{N}_v: v\in \Lambda_N\}$. Further, let $\mathcal C_\star^{\Lambda_N} = \mathcal C^{\Lambda_N} \cap \hat {\mathcal C}^{\Lambda_N}$ (so $\mathcal C_\star^{\Lambda_N}$ is a version of $\mathcal C_*^{\Lambda_N}$, but it replaces $\tilde {\mathcal C}^{\Lambda_N}$ with $\hat {\mathcal C}^{\Lambda_N}$ in its definition).
\begin{lemma}\label{lem-zero-m-annulus}
Let $\Delta(N) = (N/4)^{-\alpha(\alpha')^2}$ and define $
\{\hat h^{(N)}_v: v\in \Lambda_N\}$ as in \eqref{eq-zero-def-tilde-h-annulus}.
Then there exists $C=C(\epsilon)>0$ such that $\P(o\in \mathcal C_\star^{\Lambda_N}) \leq C N^{-5}$.
\end{lemma}
\begin{proof}
For $v\in \partial \Lambda_{N/2}$, let $B_v$ be a translated copy of $\Lambda_{N/4}$ centered at $v$.
Thus, for all $u \in B_v$ we have $\hat h^{(N)}_u = h_u + (N/4)^{-\alpha(\alpha')^2}$. Recall $m^\diamond_{N/4}(N/4)$ as in Lemma~\ref{lem-zero-m-star}.  By \eqref{eq-zero-monotonicity} and Lemma~\ref{lem-zero-m-star},
$$\P(v\in \mathcal C_\star^{\Lambda_N}) \leq m^\diamond_{N/4}(N/4) \leq C N^{-6}\,.$$
 Hence,
$\P(\partial \Lambda_{N/2} \cap \mathcal C_\star^{\Lambda_N} \neq \emptyset) \leq C N^{-5}$ by a simple union bound. Combined with Lemma~\ref{lem-zero-percolation-property} (and the simple observation that $o$ cannot percolate in $\mathcal C_\star^{\Lambda_N}$ to $\partial \Lambda_N$ if $\partial \Lambda_{N/2} \cap \mathcal C_\star^{\Lambda_N} = \emptyset$), this completes the proof of the lemma.
\end{proof}

\begin{lemma}\label{lem-zero-m-N-bound}
There exists $C= C(\epsilon)>0$ such that  $m_N \leq C N^{-3}$.
\end{lemma}
\begin{proof}
A rough intuition behind the proof is as follows: the random field in each dyadic annulus has probability close to 1 to stop the event $\{o\in \mathcal C^{\Lambda_N}\}$ from occurring and thus altogether we get a polynomial upper bound with large power. In order to formalize the proof, we will apply Lemma~\ref{lem-zero-m-annulus} and employ a careful analysis to justify the ``independence'' among different scales.

Without loss of generality, let us only consider  $N = 4^n$ for some $n\geq 1$. For each such $N$,  define $
\{\hat h^{(N)}_v: v\in \Lambda_N\}$ as in \eqref{eq-zero-def-tilde-h-annulus} with $\Delta(N)= (N/4)^{-\alpha(\alpha')^2}$.
 Let $E_\ell = \{o\not\in \mathcal C_\star^{\Lambda_{4^{\ell}}}\}$ and $E = \cap_{0.9n \leq \ell \leq n} E_\ell$. (Note that there is no containment relation among the events $E_\ell$'s, since each event depends on a different perturbation.) By Lemma~\ref{lem-zero-m-annulus}, we see that $\P(E^c) \leq CN^{-3}$ for some $C = C(\epsilon)>0$ (whose value may be adjusted later in the proof). Write $\mathfrak A_{\ell} = \Lambda_{4^{\ell}} \setminus \Lambda_{4^{\ell-1}}$. For $0.9 n\leq \ell \leq n$, let $\mathcal F_\ell = \sigma(h_v: v\in \Lambda_{4^\ell})$ and write
\begin{equation}\label{eq-zero-Gaussian-conditioning}
h_v = (|\mathfrak A_{\ell}|)^{-1} h_{\mathfrak A_{\ell}} + g_v \mbox{ for } v\in \mathfrak A_{\ell}\,,
\end{equation}
where $\{g_v: v \in \mathfrak A_{\ell}\}$ is a mean-zero Gaussian process independent of $h_{\mathfrak A_{\ell}}$ and $\{g_v: v \in \mathfrak A_{\ell}\}$ for $0.9 n\leq \ell \leq n$ are mutually independent (note that $g_v$'s are linear combinations of a Gaussian process and their means and covariances can be easily computed). Let $\mathcal F'_\ell$ be the $\sigma$-field which contains every event in $\mathcal F_{\ell}$ that is independent of $h_{\mathfrak A_{\ell}}$ (so in particular $\mathcal F_\ell \subset \mathcal F'_{\ell+1} \subset \mathcal F_{\ell+1}$). By monotonicity,  there exists an interval $I_\ell$ measurable with respect to $\mathcal F'_{\ell}$ such that conditioned on $\mathcal F'_\ell$ we have $o \in \mathcal C^{\Lambda_{4^{\ell}}}$ if and only if $h_{\mathfrak A_{\ell}} \in I_\ell$. Let $I'_\ell$ be the maximal sub-interval of $I_\ell$ which shares the upper endpoint and with length $|I'_\ell| \leq  \frac{|\mathfrak A_{\ell}| \cdot 4^{\alpha (\alpha')^2}}{4^{ \alpha (\alpha')^2 \ell}}$.
By our definition of $E_{\ell}$, we see
 from \eqref{eq-zero-Gaussian-conditioning} that  conditioned on $\mathcal F'_\ell$ we have $\{o \in \mathcal C^{\Lambda_{4^{\ell}}}\} \cap E_{\ell}$ only if $h_{\mathfrak A_{\ell}} \in I'_\ell$. Thus, for $0.9 n \leq \ell \leq n$,
$$\P(\{o \in \mathcal C^{\Lambda_{4^{\ell}}}\} \cap E_{\ell} \mid \mathcal F'_\ell) \leq \P(h_{\mathfrak A_{\ell}}\in I'_\ell)\,.$$
Combined with the fact that $\var (h_{\mathfrak A_{\ell}}) = \epsilon^2 |\mathfrak A_{\ell}|$, this gives that
$$\P(\{o \in \mathcal C^{\Lambda_{4^{\ell}}}\} \cap E_{\ell} \mid \mathcal F'_\ell) \leq \frac{C}{4^{\ell(\alpha (\alpha')^2-1)}}\,.$$
Since $\{o \in \mathcal C^{\Lambda_{4^{n}}}\} \cap E = \cap_{\ell = 0.9n}^{n} (\{o \in \mathcal C^{\Lambda_{4^{\ell}}}\} \cap E_{\ell})$ and since $\{o \in \mathcal C^{\Lambda_{4^{\ell}}}\} \cap E_{\ell}$ is $\mathcal F_\ell$-measurable (and thus is $\mathcal F'_{\ell+1}$-measurable), we deduce that $\P(\{o\in \mathcal C^{\Lambda_N}\} \cap E) \leq C N^{-3}$. Combined with the fact that $\P(E^c) \leq CN^{-3}$, it completes the proof of the lemma.
\end{proof}

\begin{proof}[Proof of Theorem~\ref{thm-main-zero}]
Let $N_0 = N_0(\epsilon)$ be chosen later. For $B\in \mathcal B(N, N_0)$, we say $B$ is open if $\mathcal C^{B^{\mathrm{large}}} \cap B \neq \emptyset$. Clearly, this percolation process satisfies the $(N, N_0, 4, p)$-condition where
\begin{equation}\label{eq-zero-p-bound}
p  = \P(\mathcal C^{B^{\mathrm{large}}} \cap B \neq \emptyset) \leq N_0^2 m_{N_0/2} \leq CN_0^{-1}  \mbox{ for } C = C(\epsilon)>0\,.
\end{equation}
(The last transition above follows from Lemma~\ref{lem-zero-m-N-bound}.)
In addition, we note that in order for $o\in \mathcal C^{\Lambda_N}$, it is necessary that there exists an open lattice animal on $B\in \mathcal B(N, N_0)$ with size at least $\frac{N}{10N_0}$. Now, choosing $N_0$ sufficiently large (so that $p$ is sufficiently small, by \eqref{eq-zero-p-bound}) and applying Lemma~\ref{lem-zero-enhance} completes the proof.
\end{proof}

\section{Exponential decay at positive temperatures}\label{sec:positive-temperature}

In this section, we prove Theorem~\ref{thm-main} for the case of $T>0$.
Our proof method follows the basic framework presented in Section~\ref{sec:zero-temperature} for the case of $T=0$, which applies the result in \cite{AB99} in a crucial way. However, there seem to be significant additional obstacles due to the randomness of Ising measures at  positive temperatures.
For $T = 0$, it suffices to consider the ground state which is unique with probability 1, and thus ground states with different boundary conditions and external fields are naturally coupled together. In the case of $T>0$, on the one hand we try to carry out our analysis with validity for all reasonable (e.g., for all monotone couplings) couplings of  Ising measures whenever possible (see Section~\ref{sec-perturbation}); on the other hand it seems necessary to construct a coupling with some desirable properties in order to apply \cite{AB99} (see Section~\ref{sec-admissible-coupling}). Both of these require some new ideas as well as some delicate treatment.

Organization for the rest of this section is as follows. In Section~\ref{sec-perturbation}, we verify the hypothesis in \cite{AB99} via a perturbation argument and thereby prove that under any monotone coupling for Ising spins with plus/minus boundary conditions, the intrinsic distance for the induced graph on vertices with disagreements has dimension strictly larger than 1. The proof method is inspired by that of Proposition~\ref{prop-zero-crossing-dimension}, but the implementation is largely different with new tricks involved. In Section~\ref{sec-admissible-coupling}, we introduce the notion of adaptive admissible coupling and a multi-scale construction of an adaptive admissible coupling is then given in Section~\ref{sec-contruction-adaptive-coupling}. In Section~\ref{sec-another-perturbation}, we then introduce another perturbation argument, using which we analyze our adaptive admissible coupling in Section~\ref{sec-analysis-coupling} and prove a crucial estimate in Lemma~\ref{lem-bound-m}. In Section~\ref{sec-proof-main-thm}, we provide the proof of Theorem~\ref{thm-main} for $T>0$, which requires to employ an admissible coupling such that the disagreement percolates to the boundary.

\subsection{Intrinsic distance on disagreements via a perturbation argument} \label{sec-perturbation}

For any $A\subset \mathbb Z^2$, we continue to denote by $d_A (\cdot, \cdot)$ the intrinsic distance on $A$, i.e., the graph distance on the induced subgraph on $A$.  Let $\sigma^{\Lambda_N,  \pm}$ be spins sampled according to  $\mu^{\Lambda_N,  \pm}$. We will continue to use repeatedly the standard monotonicity properties of the Ising model with respect to external fields and boundary conditions (c.f. \cite[Section 2.2]{AP18} for detailed discussions). Let $\pi$ be a monotone coupling of $\mu^{\Lambda_N, \pm}$ (that is, under $\pi$ we have $\sigma^{\Lambda_N, +} \geq \sigma^{\Lambda_N, -}$) and let
\begin{equation}\label{eq-def-mathcal-C}
\mathcal C^{\Lambda_N} = \mathcal C^{\Lambda_N, \pi} = \{v\in \Lambda_N: \sigma^{\Lambda_N, +}_v > \sigma^{\Lambda_N, -}_v\}\,.
\end{equation}
(Note that $\pi$ depends on the random field $h$.)
In addition, denote by $\P \otimes \pi$ the joint measure of the external fields and the spin configurations (similar notations also apply below). The following proposition is the major goal of this section.
\begin{prop}\label{prop-crossing-dimension}
There exist $\alpha = \alpha(\epsilon, \beta) > 1$, $\kappa = \kappa(\epsilon, \beta)>0$ such that the following holds. For all $0<c\leq 1$, there exists $N_0 = N_0(\epsilon, \beta, c)$ such that for all $N\geq N_0$ and $1\leq N_1 \leq N_2 \leq N/2$ with $N_2-N_1 \geq N^c$  the following holds for all monotone coupling $\pi$ of $\mu^{\Lambda_N, \pm}$:
\begin{equation}\label{eq-box-counting}
\P \otimes \pi(d_{\mathcal C^{\Lambda_N}} (\partial \Lambda_{N_1}, \partial \Lambda_{N_2}) \leq (N_2-N_1)^{\alpha}) \leq \kappa^{-1} e^{- N^{\kappa c}}\,.
\end{equation}
\end{prop}
\begin{remark}
(1) The preceding proposition is analogous to Proposition~\ref{prop-zero-crossing-dimension}. In the present case, it is crucial that the result holds for all monotone couplings (note that the intrinsic distance may depend on the coupling), so that we can apply it to couplings which we construct later.

(2) In Proposition~\ref{prop-crossing-dimension}, we introduce parameters $N_1, N_2$ (as opposed to $N_1 = N/4$ and $N_2 = N/2$ in Proposition~\ref{prop-zero-crossing-dimension}) for convenience of later applications. The condition that $N_2-N_1 \geq N^c$ is just to ensure that the decay in probability absorbs the number of choices for starting and ending points of the shortest path. This slight extension does not introduce complication to the proof.
\end{remark}
The proof of Proposition~\ref{prop-crossing-dimension} again crucially relies on the result of \cite{AB99}. In order to apply \cite{AB99}, the following lemma (analogous to Lemma~\ref{lem-zero-bound-hard-crossing}) is a key ingredient. For any annulus $\mathcal A$ and  $\mathcal C\subset \mathbb Z^2$, we continue to denote by $\mathrm{Cross}_{\mathrm{hard}}(\mathcal A, \mathcal C)$ the event that there is a contour in $\mathcal C$ which separates the inner and outer boundaries of $\mathcal A$. Let
\begin{equation}\label{eq-def-mathcal-E-pm}
\mathcal E^\pm = \mathcal E^\pm_N = \mathrm{Cross}_{\mathrm{hard}}(\Lambda_{N/8} \setminus \Lambda_{N/32}, \{v \in \Lambda_N: \sigma^{\Lambda_N, \pm}_v = \pm 1\})\,.
\end{equation}
\begin{lemma}\label{lem-crossing-key}
There exists $\delta=\delta(\epsilon, \beta)>0$ such that   for all  $N\geq 32$
 $$ \min\{\mathbb{P}\otimes\mu^{\Lambda_N, +}(\mathcal E^{+}), \P (\mbox{$\sum_{v\in \Lambda_{N/8}}$} (\langle  \sigma^{\Lambda_N, +}_v\rangle_{\mu^{\Lambda_N, +}} - \langle\sigma^{\Lambda_N, -}_v\rangle_{\mu^{\Lambda_N, -}}) > 10^{-3} N)\} \leq 1 - \delta\,.$$
 In particular, $\mathbb{P}\otimes \pi(\mathrm{Cross}_{\mathrm{hard}}(\Lambda_{N/8}\setminus\Lambda_{N/32}, \mathcal{C}^{\Lambda_N}))\leq 1-\delta$ for all monotone coupling $\pi$ of $\mu^{\Lambda_N, \pm}$.
 \end{lemma}

\subsubsection{A perturbative analysis}
Before proving Lemma~\ref{lem-crossing-key}, we need some preparational work on a certain perturbative analysis. This is analogous to Lemma~\ref{lem-zero-perturbation}, which has been applied twice in the case of $T=0$: in the proof of Lemma~\ref{lem-zero-bound-hard-crossing} and the proof of Lemma~\ref{lem-zero-m-star}. For $T>0$, it is more complicated and thus we provide two separate versions of perturbative analysis, both of which are proved via keeping track of the free energy. The first version is presented in  Lemma~\ref{lem-free-energy-differences-2} in the present section (for the application in Lemma~\ref{lem-crossing-key}), and the second version is presented in Section~\ref{sec-another-perturbation} (for the application in Lemma~\ref{lem-bound-m}).

For any set $\Lambda \subset \mathbb Z^2$ and a configuration $\tau \in \{-1, 1\}^{\partial \Lambda}$, analogous to \eqref{eq-def-H} we can define the Hamiltonian on $\Lambda$ with boundary condition $\tau$ and external field $\{h_v\}$ by:
\begin{equation}\label{eq-def-H-general}
H^{\Lambda, \tau}(\sigma) = - \big(\sum_{u\sim v, u, v\in \Lambda} \sigma_u \sigma_v  + \sum_{u\sim v, u\in  \Lambda, v\in \partial \Lambda} \sigma_u \tau_v + \sum_{u\in \Lambda} \sigma_u h_u\big) \mbox{ for } \sigma \in \{-1, 1\}^{\Lambda}\,.
\end{equation}
We can then analogously define the Ising measure $\mu^{\Lambda, \tau}$ by assigning probability to $\sigma\in \{-1, 1\}^\Lambda$ proportional to $e^{-\beta H^{\Lambda, \tau}(\sigma)}$. In addition, we define the corresponding log-partition-function (it is the negative of the free energy; in our analysis, it seems cleaner to work with the log-partition-function so not to be confused by the negative sign)
\begin{equation}\label{eq-free}
    F^{\Lambda, \tau}= \frac{1}{\beta}\log \big(\sum_{\sigma\in \{-1,1\}^\Lambda} e^{-\beta H^{\Lambda, \tau}(\sigma)}\big).
\end{equation}

 For simplicity, we will only consider $N = 2^n$ for $n\geq 10$.
 For $\Delta > 0$, $\Delta'\geq 0$ and $0\leq t\leq 1$,  we will consider the following perturbed field in this section (which is increasing in $t$):
\begin{equation}\label{eq-def-tilde-h}
h^{(t)}_v = h^{(t, N)}_v  = \begin{cases}
h_v + \Delta', &\mbox{ for } v\in \Lambda_{N}\setminus \Lambda_{N/8}\,,\\
h_v + t\Delta, &\mbox{ for } v\in \Lambda_{N/8}\,.
\end{cases}
\end{equation}
(We draw the reader's attention to that $t$ appeared in the definition of $h^{(t)}_v$ only for $v\in \Lambda_{N/8}$, and that $h^{(0)} \neq h$ if $\Delta'>0$. The perturbation in \eqref{eq-def-tilde-h} is more subtle than that in \eqref{eq-zero-def-tilde-h}, for the reason that we wish to take advantage of \eqref{eq-Delta'} below later with a judicious choice of $\Delta'$.)
Let $\mu^{\Lambda_N,  \pm, t}$ be Ising measures with plus/minus boundary conditions and external field $\{h^{(t)}_v: v\in \Lambda_N\}$. In addition, let $H^{\Lambda_N, \pm, t}$ be the corresponding Hamiltonians, let $F^{\Lambda_N, \pm, t}$ be the corresponding log-partition-functions, and let $\sigma^{\Lambda_N,  \pm, t}$ be spin configurations sampled according to $\mu^{\Lambda_N,  \pm, t}$.

For notation convenience, for any set $\Gamma\subset \mathbb Z^2$, let $S_\Gamma$ be the collection of vertices which are not in $\Gamma$ and are separated by $\Gamma$ from $\infty$ on $\mathbb Z^2$ (i.e., the collection of vertices that are enclosed by $\Gamma$).

Let $S \subset \Lambda_N$ be a subset which contains $\Lambda_{N/8}$ and let $\Gamma = \partial S$ (thus we have $S \subset S_\Gamma$).
For any $\tau\in \{-1, 1\}^\Gamma$, we denote by $\mu^{S, \tau, t}$ the Ising measure on $S$ with boundary condition $\tau$ and external field $\{h_v^{(t)}: v\in S\}$. In addition,  let $H^{S, \tau, t}$ be the Hamiltonian for the corresponding Ising spin, and let $F^{S, \tau, t}$  be the corresponding log-partition-function. Also, we let $\sigma^{S, \tau, t}$  be the spin configuration sampled according to  $\mu^{S, \tau, t}$. For later applications, it would be useful to consider the log-partition-function restricted to a subset of configurations. To this end, we define
\begin{equation}\label{eq-free-Omega}
    F^{S, \tau, t}_\Omega=  \frac{1}{\beta}\log \big(\sum_{\sigma\in \Omega} e^{-\beta H^{S, \tau, t}(\sigma)}\big) \mbox{ for } \Omega \subset \{-1,1\}^S.
\end{equation}
In addition, for any measure $\mu^{S, \tau, t}$, we define $\mu^{S, \tau, t}_{\Omega}$ to be a measure such that
$$\mu^{S, \tau, t}_\Omega(\sigma) = (\mu^{S, \tau, t}(\Omega) )^{-1}\mu^{S, \tau, t}(\sigma) \mbox{ for }\sigma\in \Omega\,.$$ (We draw readers' attention to that $\mu^{S, \tau, t}(\Omega)$ is the total measure of $\Omega$ under $\mu^{S, \tau, t}$ and thus is a number, and that $\mu^{S, \tau, t}_\Omega$ is the measure $\mu^{S, \tau, t}$ conditioned on the occurrence of $\Omega$.) For convenience, we let $\sigma^{S, \tau, t}_{\Omega}$ be the spin configuration sampled according to $\mu^{S, \tau, t}_{\Omega}$.
Further, define (note that below we sum over $v\in \Lambda_{N/32}$ as opposed to $v\in S$)
 \begin{equation}\label{eq-def-m-S-zeta-t}
 m^{S, \tau, t}_{\Omega} = \sum_{v\in \Lambda_{N/32}} \langle \sigma^{S, \tau, t}_{\Omega, v} \rangle_{\mu^{S, \tau, t}_{\Omega}}\,.
\end{equation}
For notation convenience, we write $ m^{S, \tau, t} =  m^{S, \tau, t}_{\Omega}$ if $\Omega = \{-1, 1\}^S$.
We say $\Omega\subset \{-1, 1\}^S$ is an increasing set if $\sigma\in \Omega$ implies that $\sigma'\in \Omega$ provided $\sigma' \geq \sigma$, and we say $\Omega$ is a decreasing set if $\Omega^c$ is an increasing set. In what follows, we consider $\tau^+, \tau^- \in \{-1, 1\}^\Gamma$ such that $\tau^+ \geq \tau^-$.
 \begin{lemma}\label{lem-free-energy-differences-2}
Quench on the external field $\{h_v\}$. We have that for any increasing set $\Omega^+\subset \{-1, 1\}^S$ and any decreasing set $\Omega^-\subset \{-1, 1\}^S$
$$
\Delta\int_0^1 (m^{S,  \tau^+, t}_{\Omega^+} - m^{S, \tau^-, t}_{\Omega^-}) dt
\leq 8\sum_{v\in \Gamma}(\tau^+_v - \tau^-_v) - \frac{1}{\beta}\big( \log \mu^{S, \tau^+, 0}(\Omega^+)+ \log \mu^{S, \tau^-, 1}(\Omega^-) \big) \,.$$
 \end{lemma}
 \begin{proof}
 The proof is done via keeping track of the change on the difference of log-partition-functions with respect to different boundary conditions when we perturb the external field. In {\bf Step 1}, we bound such difference from above by the number of disagreements on boundary conditions; in {\bf Step 2} we bound such difference from below by the expected number of disagreements, with a caveat that we use the notion of ``restricted'' log-partition-functions as in \eqref{eq-free-Omega}; in {\bf Step 3}, we address the caveat by linking the two notions of log-partition-functions.

\noindent {\bf Step 1.} We will prove (below the equality is obvious since $\tau^+ \geq \tau^-$)
 \begin{equation}\label{eq-free-energy-upper-2}
  (F^{S,  \tau^+,1} - F^{S,  \tau^-, 1}) - (F^{S,  \tau^+,0} - F^{S,  \tau^-, 0}) \leq 16\cdot \#\{v\in \Gamma:  \tau^+_v \neq  \tau^-_v\} =  8 \sum_{v\in \Gamma}(\tau^+_v - \tau^-_v)\,.
 \end{equation}
  (Here we use $\#A$ to denote the cardinality of $A$ for a finite set $A$. We switch from  the more compact notation $|A|$ to $\#A$ in this section, as we wish to avoid somewhat awkward notation when $|$ is followed by another $\mid$ which means ``conditioned on''.)
 Since each vertex has 4 neighbors in $\mathbb Z^2$,  a straightforward computation gives that
 \begin{align*}
F^{S,  \tau^+,1} - F^{S,  \tau^-, 1}
  = \frac{1}{\beta}\log \frac{\sum_{\sigma}e^{-\beta  H^{S, \tau^+, 1}(\sigma)}}{\sum_{\sigma}e^{-\beta H^{S,  \tau^-, 1}(\sigma) }}
\leq\frac{1}{\beta}\log e^{8 \beta \cdot \#\{v\in \Gamma: \tau^+_v \neq  \tau^-_v\}} \leq 8\cdot \#\{v\in \Gamma:  \tau^+_v \neq  \tau^-_v\}\,.
 \end{align*}
 Similarly, we have that $F^{S,    \tau^+,0} - F^{S, \tau^-, 0}\geq - 8\cdot\#\{v\in \Gamma: \tau^+_v \neq  \tau^-_v\}$. This proves \eqref{eq-free-energy-upper-2}.

\noindent {\bf Step 2.} We will prove
 \begin{equation}
 \label{eq-free-energy-lower-2}
 (F^{S,  \tau^+, 1}_{\Omega^+} - F^{S,  \tau^-, 1}_{\Omega^-}) - (F^{S, \tau^+, 0}_{\Omega^+} -  F^{S, \tau^-,0}_{\Omega^-})  \geq \Delta\int_0^1 (m^{S,  \tau^+, t}_{\Omega^+} - m^{S, \tau^-, t}_{\Omega^-}) dt\,.
 \end{equation}
  We write
 \begin{equation}\label{eq-int-2}
 (F^{S,  \tau^+, 1}_{\Omega^+} - F^{S,  \tau^-, 1}_{\Omega^-}) - (F^{S, \tau^+, 0}_{\Omega^+} -  F^{S, \tau^-,0}_{\Omega^-})  = (F^{S,  \tau^+, 1}_{\Omega^+} - F^{S, \tau^+, 0}_{\Omega^+} ) - (F^{S,  \tau^-, 1}_{\Omega^-} -  F^{S, \tau^-,0}_{\Omega^-}) \,.
 \end{equation}
Thus, we get that
 \begin{equation}\label{eq-part-2}
F^{S,  \tau^+, 1}_{\Omega^+} - F^{S, \tau^+, 0}_{\Omega^+} =\int_0^1 \frac{d F^{S,   \tau^+, t}_{\Omega^+}}{dt} dt,\quad \quad
F^{S,  \tau^-, 1}_{\Omega^-} -  F^{S, \tau^-,0}_{\Omega^-} =\int_0^1 \frac{d F^{S,  \tau^-, t}_{\Omega^-}}{dt}  dt\,.
 \end{equation}
Since $\frac{d F^{S,   \tau^+, t}_{\Omega^+}}{dt}  = \sum_{v\in \Lambda_{N/8}} \Delta \langle \sigma^{S,  \tau^+, t}_{\Omega^+, v} \rangle_{\mu^{S,  \tau^+, t}_{\Omega^+}}$ and $\frac{d F^{S,    \tau^-, t}_{\Omega^-}}{dt}  = \sum_{v\in \Lambda_{N/8}} \Delta \langle \sigma^{S,  \tau^-, t}_{\Omega^-, v} \rangle_{\mu^{S, \tau^-, t}_{\Omega^-}}$, we see
 $$\frac{d F^{S,   \tau^+, t}_{\Omega^+}}{dt}  -  \frac{d F^{S,  \tau^-, t}_{\Omega^-}}{dt}  \geq \sum_{v\in \Lambda_{N/32}} \Delta( \langle \sigma^{S,  \tau^+, t}_{\Omega^+, v} \rangle_{\mu^{S, \tau^+, t}_{\Omega^+}} -  \langle \sigma^{S,  \tau^-, t}_{\Omega^-, v} \rangle_{\mu^{S, \tau^-, t}_{\Omega^-}}) = \Delta m^{S, \tau^+, t}_{\Omega^+} - \Delta m^{S, \tau^-, t}_{\Omega^-}\,,$$
 where the inequality follows from the fact that
 $$\langle \sigma^{S, \tau^+, t}_{\Omega^+, v} \rangle_{\mu^{S, \tau^+, t}_{\Omega^+}} \geq  \langle \sigma^{S, \tau^+, t}_v \rangle_{\mu^{S, \tau^+, t}} \geq \langle \sigma^{S,  \tau^-, t}_v \rangle_{\mu^{S, \tau^-, t}} \geq \langle \sigma^{S,  \tau^-, t}_{\Omega^-, v} \rangle_{\mu^{S, \tau^-, t}_{\Omega^-}}\mbox{ for all }v\in S\,.$$
In the preceding display, the first and the third inequalities follow from FKG inequality \cite{FKG} and the second inequality follows from monotonicity.
 Combined with \eqref{eq-part-2} and \eqref{eq-int-2}, it yields \eqref{eq-free-energy-lower-2}.

\noindent {\bf Step 3.} From definitions as in \eqref{eq-free} and \eqref{eq-free-Omega}, we see that
\begin{equation}\label{eq-free-energy-comparison}
F^{S,  \tau^+, 1}  - F^{S,  \tau^+, 1}_{\Omega^+} = -\frac{1}{\beta} \log \mu^{S, \tau^+, 1}(\Omega^+)\,,
\end{equation}
and similar equalities hold for other combinations of boundary conditions, external fields and $\Omega^\pm$.

Combining \eqref{eq-free-energy-upper-2}, \eqref{eq-free-energy-lower-2} and \eqref{eq-free-energy-comparison}, we complete the proof of the lemma.
 \end{proof}

\subsubsection{A lower bound on the intrinsic distance}
Denote by $\mathcal V^{\sigma, \pm} = \{v\in S: \sigma_v = \pm 1\}$  for $S\subset \Lambda_N$ and $\sigma\in \{-1, 1\}^{S}$. For any $S\supset \Lambda_{N/8}$, define
\begin{equation}\label{eq-def-Omega-pm}
\Omega^\pm = \Omega^\pm(S)= \{\sigma\in \{-1, 1\}^{S}: \mathrm{Cross}_{\mathrm{hard}}(\Lambda_{N/8}\setminus\Lambda_{N/32}, \mathcal{V}^{\sigma, \pm}) \mbox{ occurs }\}\,.
\end{equation}
We see that $\Omega^+$ is an increasing set and $\Omega^-$ is a decreasing set. For $A\subset \Lambda \subset \mathbb Z^2$ and $\sigma\in \{-1, 1\}^\Lambda$, we denote by $\sigma_A$ the restriction of $\sigma$ on $A$.
 Let $r>0$ be a constant chosen later. Recall \eqref{eq-def-tilde-h}. Let $\Delta = \frac{10^{10}r^8}{N (\beta \wedge 1)}$ and $\Delta' = t^*\Delta$ for $0\leq t^*\leq 1$ to be chosen.
 \begin{lemma}\label{lem-contiguous}
 For any $p, r >0$, there exists $c = c(\epsilon,p, r, \beta)>0$ such that for any event $E_N$ with $\P(\{h^{(t)}_v: v\in \Lambda_N\}\in E_N)\geq p$ for some $0\leq t, t^*\leq 1$, we have that
$\P(\{h_v: v\in \Lambda_N\}\in E_N) \geq c$.
 \end{lemma}
 \begin{proof}
The proof is an adaption of Lemma~\ref{lem-zero-perturbation-contiguous} except for minimal  notation change, and thus we omit further details.
 \end{proof}

 \begin{proof}[Proof of Lemma~\ref{lem-crossing-key}]
 The proof shares similarity with that of Lemma~\ref{lem-zero-bound-hard-crossing}, but the present proof is substantially more involved.
 We first provide a heuristic outline of the proof, and we will not be precise on notations or unimportant constants in this informal description. The statement will follow immediately if the probability for existence of a plus contour with respect to plus boundary condition is strictly less than 1, and thus we suppose otherwise (formally, we suppose \eqref{eq-crossing-annulus-1} below). We wish to compare the number of disagreements in $\Lambda_{N/32}$ with that in $\mathcal A_{N/2}$. To this end, it will be useful to consider the ``enhanced'' disagreements in $\Lambda_{N/32}$ (that is, when we pose plus and minus boundary conditions on $\partial \Lambda_{N/8}$ instead of $\partial \Lambda_N$; the word ``enhanced'' is chosen because by monotonicity the enhanced disagreements stochastically dominate the original disagreements). We now compare the enhanced disagreements in $\Lambda_{N/32}$ and disagreements in $\mathcal A_{N/2}$ in both directions.
 \begin{itemize}
 \item The ``$\leq$'' direction ({\bf Step 1} below): This is where plus (minus) contours come into play. Conditioned on existence of plus and minus contours, the disagreements in $\Lambda_{N/32}$ stochastically dominate the enhanced disagreements. In addition, by Lemma~\ref{lem-free-energy-differences-2}, the number of disagreements in $\Lambda_{N/32}$ is upper bounded by that in $\mathcal A_{N/2}$ (up to an additive term that is related to the probability of existence of plus/minus contours, which we will address later). Altogether, we get that the number of enhanced disagreements in $\Lambda_{N/32}$ is upper bounded by the number of disagreements in $\mathcal A_{N/2}$ (see \eqref{eq-Step-1}).
 \item The ``$\geq$'' direction ({\bf Step 2} below): The set of disagreements in $\mathcal A_{N/2}$ is dominated by a union of constant copies of enhanced disagreements in $\Lambda_{N/32}$, where all these copies are independent with the enhanced disagreements in $\Lambda_{N/32}$. This implies that with positive probability, the number of enhanced disagreements in $\Lambda_{N/32}$ is larger (up to a constant factor) than the number of disagreements in $\mathcal A_{N/2}$ (see \eqref{eq-Step-2-final}).
 \end{itemize}
 Now, if we choose the constants appropriately, we will see that the preceding two scenarios will occur simultaneously with positive probability, which yield bounds in two directions that ``almost'' contradict each other. In order for no contradiction, the only possibility now is that the logarithmic term we ignored earlier (which becomes $\frac{N}{2\beta}$ in \eqref{eq-Step-1}) plays a significant role. But this can happen only when the typical number of enhanced disagreements is at most of order $N$, in which case an application of Markov's inequality (see \eqref{eq-Markov-inequality}) yields the desired lemma.

 We next carry out the proof formally, where we slightly shuffle the order of arguments: we first show that if the typical number of enhanced disagreements is at most of order $N$ (see \eqref{eq-theta*}), then the lemma holds. Next, we prove \eqref{eq-theta*} (which is the main challenge) by contradiction, via the aforementioned two directional comparisons.

  For convenience of notation, write $$\mathcal E^{\pm, t}=  \mathrm{Cross}_{\mathrm{hard}}(\Lambda_{N/8}\setminus\Lambda_{N/32}, \mathcal{V}^{\sigma^{\Lambda_N, \pm, t}, \pm})\,.$$
 We suppose that
 \begin{equation}\label{eq-crossing-annulus-1}
\min_{0\leq t\leq 1}\{\mathbb{P}\otimes\mu^{\Lambda_N, +, t}(\mathcal E^{+, t}), \mathbb{P}\otimes\mu^{\Lambda_N, -, t}(\mathcal E^{-, t})\} \geq 1-r^{-4}10^{-10} \,.
\end{equation}
Otherwise Lemma~\ref{lem-crossing-key} follows from Lemma~\ref{lem-contiguous} (since under any monotone coupling we have $\mathrm{Cross}_{\mathrm{hard}}(\Lambda_{N/8}\setminus\Lambda_{N/32}, \mathcal{C}^{\Lambda_N}) \subset \mathcal E^+_N \cap \mathcal E^-_N$, where $\mathcal E^\pm_N$ is defined in \eqref{eq-def-mathcal-E-pm}). We remark that by monotonicity the preceding inequality is equivalent to $\min\{\mathbb{P}\otimes\mu^{\Lambda_N, +, 0}(\mathcal E^{+, 0}), \mathbb{P}\otimes\mu^{\Lambda_N, -, 1}(\mathcal E^{-, 1})\} \geq 1-r^{-4}10^{-10} $.

Let $\mathcal E^\star = \{\mu^{\Lambda_N, +, 0}(\mathcal E^{+, 0}) \geq 99/100\} \cap \{\mu^{\Lambda_N, -, 1}(\mathcal E^{-, 1}) \geq 99/100\}$ be an event measurable with respect to the Gaussian field. By \eqref{eq-crossing-annulus-1}, we see that
\begin{equation}\label{eq-prob-E-star}
\P(\mathcal E^\star)\geq 1-10^{-2}r^{-4}\,.
\end{equation}
Let $t^*\in [0, 1]$ be such that
 \begin{equation}\label{eq-Delta'}
\inf\{\theta: \P(m^{\Lambda_{N/8}, +, t^*} -m^{\Lambda_{N/8}, -, t^*}\geq \theta ) \leq 1/2r\} = \theta^*\,,
 \end{equation}
 where $\theta^* = \min_{0\leq t \leq 1} \inf \{\theta: \P(m^{\Lambda_{N/8}, +, t} -m^{\Lambda_{N/8}, -, t}\geq \theta ) \leq 1/2r\}\,.$
We claim that
\begin{equation}\label{eq-theta*}
\theta^* \leq 10^{-3} r^{-1} N\,.
\end{equation}
We first show that \eqref{eq-theta*} implies the lemma. For any box $A$, let $A^{\mathrm{Big}}$ be the concentric box of $A$ with side length 4 times that of $A$. Let $r$ be a large enough constant so that we can write $\Lambda_{N/8}=\cup_{i=1}^rA_i$, where $A_i$ is a copy of $\Lambda_{N/32}$ and $A_i$'s are disjoint such that
 $A_i^{\mathrm{Big}} \subset \Lambda_N$ for $1\leq i\leq r$. By monotonicity, we see that for each $1\leq i\leq r$
\begin{align*}
\P& (\sum_{v\in A_i} (\langle  \sigma^{\Lambda_N, +, t^*}_v\rangle_{\mu^{\Lambda_N, +, t^*}} - \langle\sigma^{\Lambda_N, -, t^*}_v\rangle_{\mu^{\Lambda_N, -, t^*}}) > \theta^*) \\
&\leq \P (\sum_{v\in A_i} (\langle  \sigma^{A_i^{\mathrm{Big}}, +, t^*}_v\rangle_{\mu^{A_i^{\mathrm{Big}}, +, t^*}} - \langle\sigma^{A_i^{\mathrm{Big}}, -, t^*}_v\rangle_{\mu^{A_i^{\mathrm{Big}}, -, t^*}})> \theta^*) \leq (2r)^{-1}\,,
\end{align*}
where the last inequality holds due to our choice of $t^*$ as in \eqref{eq-Delta'} and $\Delta' = t^* \Delta$ (thus $h^{(t^*)}_v = h_v + \Delta'$ for $v\in \Lambda_N$). Hence, a simple union bound gives that
\begin{equation}\label{eq-covering-argument}
\P(\sum_{v\in \Lambda_{N/8}} (\langle  \sigma^{\Lambda_N, +, t^*}_v\rangle_{\mu^{\Lambda_N, +, t^*}} - \langle\sigma^{\Lambda_N, -, t^*}_v\rangle_{\mu^{\Lambda_N, -, t^*}}) \leq r \theta^*) \geq \frac{1}{2}\,.
\end{equation}
By Lemma~\ref{lem-contiguous}, we get that
\begin{equation}\label{eq-small-disagreements}
\P (\sum_{v\in \Lambda_{N/8}} (\langle  \sigma^{\Lambda_N, +}_v\rangle_{\mu^{\Lambda_N, +}} - \langle\sigma^{\Lambda_N, -}_v\rangle_{\mu^{\Lambda_N, -}}) > r \theta^*) \leq 1- \delta \mbox{ for } \delta = \delta(\epsilon, \beta, r) >0\,.
\end{equation}
Note that $2\langle\#(\mathcal{C}^{\Lambda_N} \cap \Lambda_{N/8})\rangle_{\pi} = \sum_{v\in \Lambda_{N/8}} (\langle  \sigma^{\Lambda_N, +}_v\rangle_{\mu^{\Lambda_N, +}} - \langle\sigma^{\Lambda_N, -}_v\rangle_{\mu^{\Lambda_N, -}})$ on each instance of the Gaussian field for any monotone coupling $\pi$ of $\mu^{\Lambda_N, \pm}$. Therefore,  on each instance of Gaussian field (which occurs with probability at least $\delta$) such that $\sum_{v\in \Lambda_{N/8}} (\langle  \sigma^{\Lambda_N, +}_v\rangle_{\mu^{\Lambda_N, +}} - \langle\sigma^{\Lambda_N, -}_v\rangle_{\mu^{\Lambda_N, -}}) \leq r \theta^*$, we apply Markov's inequality and get that
\begin{equation}\label{eq-Markov-inequality}
 \pi(\mathrm{Cross}_{\mathrm{hard}}(\Lambda_{N/8}\setminus\Lambda_{N/32}, \mathcal{C}^{\Lambda_N})) \leq \pi(\#(\mathcal{C}^{\Lambda_N} \cap \Lambda_{N/8}) \geq \tfrac{N}{32}) \leq \frac{\theta^* r}{N/32} \leq  \frac{1}{2}\,,
\end{equation}
where the last inequality follows from \eqref{eq-theta*}. This implies that $\P \otimes  \pi(\mathrm{Cross}_{\mathrm{hard}}(\Lambda_{N/8}\setminus\Lambda_{N/32}, \mathcal{C}^{\Lambda_N})) \leq 1-\delta/2$,
 completing the proof of Lemma~\ref{lem-crossing-key} (combined with \eqref{eq-small-disagreements}).

\medskip

It remains to prove \eqref{eq-theta*}. Suppose that \eqref{eq-theta*} does not hold. We will derive a contradiction, using the following two steps.

\begin{center}
\begin{figure}[h]
  \includegraphics[width=15cm]{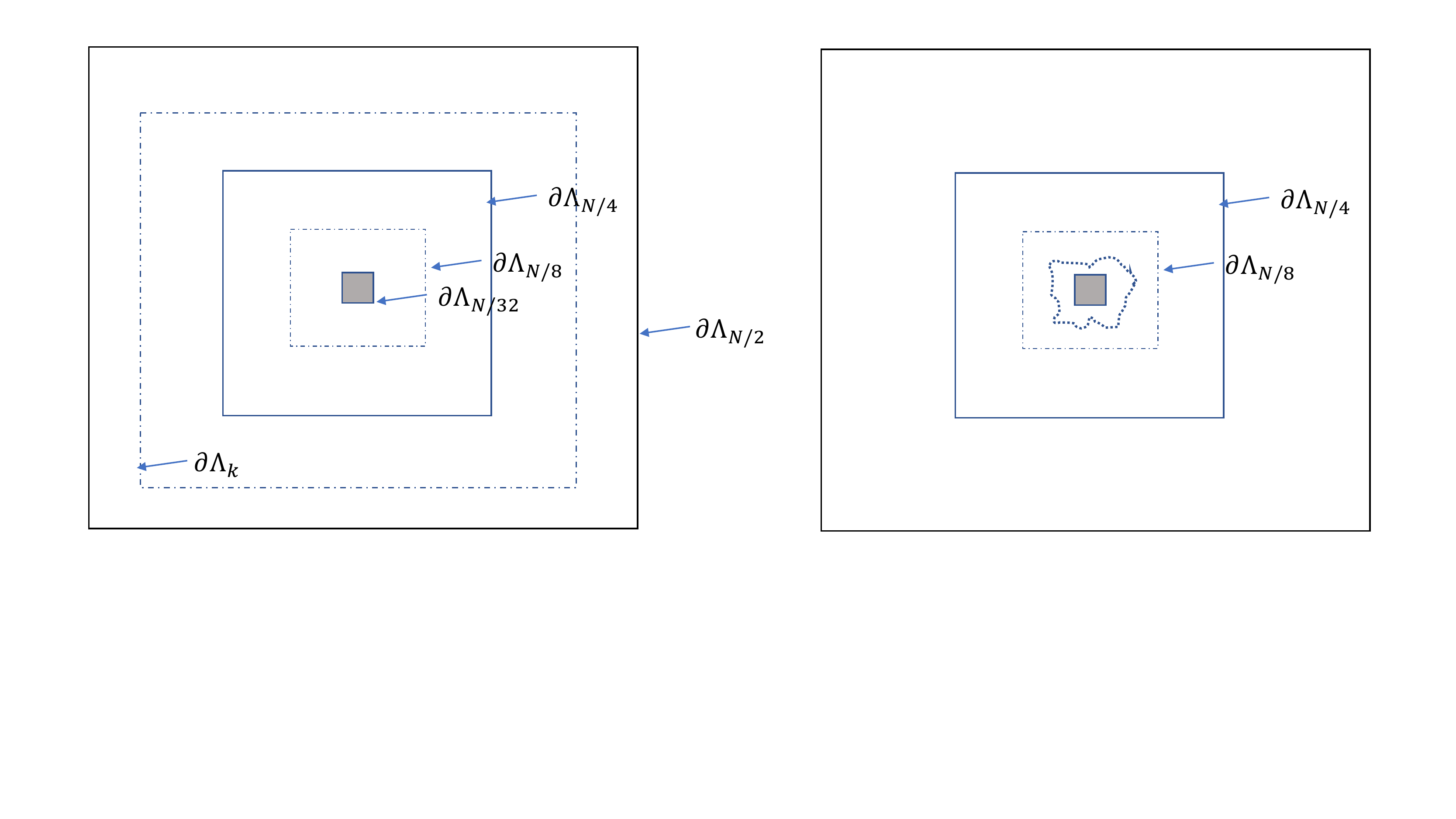}\\ \vspace{-3cm}
 \caption{Illustrations for geometric setup in Step 1 of Lemma~\ref{lem-crossing-key}. The picture on the left illustrates the setup for derivation of \eqref{eq-consequence-Lemma-2.5}, where we bound disagreements in the grey square by disagreements on $\partial \Lambda_k$ (the larger dot-line boundary). The picture on the right illustrates the setup for derivation of \eqref{eq-tau-pm-bound}: by FKG conditioned on plus (respectively minus) contour (drawn in dots in the picture) the magnetization on the grey box is pushed up (respectively down); this allows us to compare the disagreements and enhanced disagreements.
 }\label{figure-3.3}
\end{figure}
\end{center}

\noindent {\bf Step 1.} We refer to Figure~\ref{figure-3.3} for an illustration of geometric setup in this step. Fix $N/4\leq k\leq N/2$. Write $S = \Lambda_k$ and $\Gamma = \partial S$. We first quench on the Gaussian field and also condition on
\begin{equation}\label{eq-condition-on-spin}
(\sigma^{\Lambda_N, +, 1})_{\Gamma} = \tau^+\mbox{ and }(\sigma^{\Lambda_N, -, 0})_{\Gamma} = \tau^- \mbox{ where } \tau^\pm\in \{-1, 1\}^\Gamma \mbox{ and }\tau^+ \geq \tau^-\,.
\end{equation}
Applying Lemma~\ref{lem-free-energy-differences-2}, we get that (recall $\Omega^\pm = \Omega^\pm(S)$ as in \eqref{eq-def-Omega-pm})
\begin{equation}\label{eq-consequence-Lemma-2.5}
\Delta\int_0^1 (m^{S,  \tau^+, t}_{\Omega^+} - m^{S, \tau^-, t}_{\Omega^-}) dt
\leq 8\sum_{v\in \Gamma}(\tau^+_v - \tau^-_v) - \frac{1}{\beta}\big( \log \mu^{S, \tau^+, 0}(\Omega^+)+ \log \mu^{S, \tau^-, 1}(\Omega^-) \big) \,.
\end{equation}
Conditioned on $\sigma^{S, \tau^+, t}\in \Omega^+$, let $\mathfrak C \subset \mathcal V^{\sigma^{S, \tau^+, t}, +} \cap (\Lambda_{N/8}\setminus \Lambda_{N/32})$ be the outmost contour which surrounds $\Lambda_{N/32}$. Note that $\mathfrak C = \Gamma'$ is measurable with respect to $\{\sigma^{S, \tau^+, t}_v: v\in S_{\Gamma'}^c\}$. Thus, by monotonicity of Ising model we see that  $(\sigma^{S, \tau^+, t})_{\Lambda_{N/32}}$ conditioned on $\mathfrak C = \Gamma'$  stochastically dominates $(\sigma^{\Lambda_{N/8}, +,t})_{\Lambda_{N/32}}$. A similar analysis applies to $(\sigma^{S, \tau^-, t})_{\Lambda_{N/32}}$. Combined with \eqref{eq-consequence-Lemma-2.5}, it yields that
\begin{equation}\label{eq-tau-pm-bound}
\Delta\int_0^1 (m^{\Lambda_{N/8}, +, t} - m^{\Lambda_{N/8}, -, t}) dt
\leq 8\sum_{v\in \Gamma}(\tau^+_v - \tau^-_v) - \frac{1}{\beta}\big( \log \mu^{S, \tau^+, 0}(\Omega^+)+ \log \mu^{S, \tau^-, 1}(\Omega^-) \big) \,.
\end{equation}
Define $\mathcal E_{\Gamma, +} = \{\tau^+: \mu^{S, \tau^+, 0}(\Omega^+) \geq 3/4\}$ and $\mathcal E_{\Gamma, -} = \{\tau^-: \mu^{S, \tau^{-}, 1}(\Omega^-) \geq 3/4\}$. Thus,
\begin{align*}
\mu^{\Lambda_N, +, 0}(\mathcal E^{+, 0}) &= \mu^{\Lambda_N, +, 0}(\mathcal E^{+, 0} \mid (\sigma^{\Lambda_N, +, 0})_{\Gamma} \in \mathcal E_{\Gamma, +}) \mu^{\Lambda_N, +, 0}((\sigma^{\Lambda_N, +, 0})_{\Gamma} \in \mathcal E_{\Gamma, +})  \\
& \qquad+  \mu^{\Lambda_N, +, 0}(\mathcal E^{+, 0} \mid (\sigma^{\Lambda_N, +, 0})_{\Gamma} \not\in \mathcal E_{\Gamma, +}) \mu^{\Lambda_N, +, 0}((\sigma^{\Lambda_N, +, 0})_{\Gamma} \not\in \mathcal E_{\Gamma, +}) \\
&\leq \mu^{\Lambda_N, +, 0}((\sigma^{\Lambda_N, +, 0})_{\Gamma} \in \mathcal E_{\Gamma, +}) + \tfrac{3}{4} \mu^{\Lambda_N, +, 0}((\sigma^{\Lambda_N, +, 0})_{\Gamma} \not\in \mathcal E_{\Gamma, +})\,.
\end{align*}
Since $\mu^{\Lambda_N, +, 0}(\mathcal E^{+, 0}) \geq 99/100$ on $\mathcal E^\star$,
it gives that $\mu^{\Lambda_N, +, 0}((\sigma^{\Lambda_N, +, 0})_\Gamma \in \mathcal E_{\Gamma, +}) \geq 3/4$ and thus by monotonicity $\mu^{\Lambda_N, +, 1}((\sigma^{\Lambda_N, +, 1})_\Gamma \in \mathcal E_{\Gamma, +}) \geq 3/4$ (note $\mathcal E_{\Gamma, +}$ is an increasing set). Similarly, we get $\mu^{\Lambda_N, -, 0}((\sigma^{\Lambda_N, -, 0})_{\Gamma} \in \mathcal E_{\Gamma, -}) \geq 3/4$  on $\mathcal E^\star$. Consider an arbitrary monotone coupling $\pi_\Gamma$ of $\mu^{\Lambda_N, +, 1}$ and $\mu^{\Lambda_N, -, 0}$ restricted to $\Gamma$. Then we see that on $\mathcal E^\star$
$$\pi_\Gamma(\mathcal E_{\Gamma, +, - }) \geq \tfrac{3}{4} + \tfrac{3}{4} -  1=  \tfrac{1}{2}  \mbox{ where } \mathcal E_{\Gamma, +, - } = \{(\sigma^{\Lambda_N, +, 1})_\Gamma \in \mathcal E_{\Gamma, +}, (\sigma^{\Lambda_N, -, 0})_{\Gamma} \in \mathcal E_{\Gamma, -}\}\,.$$
Averaging \eqref{eq-tau-pm-bound} over the conditioning of \eqref{eq-condition-on-spin} but restricted to the event $\mathcal E_{\Gamma, +, -}$, we get that on $\mathcal E^\star$
$$\frac{\Delta}{2}\int_0^1 (m^{\Lambda_{N/8}, +, t} - m^{\Lambda_{N/8}, -, t}) dt
\leq 8\sum_{v\in \Gamma}\langle (\sigma^{\Lambda_N, +, 1}_v - \sigma^{\Lambda_N, -, 0}_v ) \one_{\mathcal E_{\Gamma, +, -}}\rangle_{\pi_\Gamma} + 2/\beta\,.$$
Since $\pi_\Gamma$ is a monotone coupling, we thus obtain that on $\mathcal E^\star$
\begin{align*}
\frac{\Delta}{2}\int_0^1 (m^{\Lambda_{N/8}, +, t} - m^{\Lambda_{N/8}, -, t}) dt
&\leq 8\sum_{v\in \Gamma}\langle \sigma^{\Lambda_N, +, 1}_v - \sigma^{\Lambda_N, -, 0}_v  \rangle_{\pi_\Gamma} + 2/\beta \\
&= 8\sum_{v\in \Gamma}(\langle \sigma^{\Lambda_N, +, 1}_v\rangle_{\mu^{\Lambda_N, +, 1}} - \langle\sigma^{\Lambda_N, -, 0}_v \rangle_{\mu^{\Lambda_N, -, 0}}) + 2/\beta \,.
\end{align*}
Summing over $N/4\leq k\leq N/2$, we deduce that on $\mathcal E^\star$
\begin{align}\label{eq-Step-1}
&8\sum_{v\in \mathcal A_{N/2}} (\langle  \sigma^{\Lambda_N, +, 1}_v\rangle_{\mu^{\Lambda_N, +, 1}} - \langle\sigma^{\Lambda_N, -, 0}_v\rangle_{\mu^{\Lambda_N, -, 0}})  +\frac{N}{2\beta}
&\geq \frac{N \Delta}{8}\int_0^1 (m^{\Lambda_{N/8}, +, t} - m^{\Lambda_{N/8}, -, t}) dt\,.
\end{align}

\noindent {\bf Step 2.}   For $N\geq 2$, recall that $\mathcal A_N = \Lambda_N \setminus \Lambda_{N/2}$ is an annulus. Adjust the value of $r$ if necessary so that we can write $\mathcal{A}_{N/2}=\cup_{i=1}^rA_i$, where $A_i$ is a copy of $\Lambda_{N/32}$ and $A_i$'s are disjoint such that
\begin{equation}\label{eq-A-i-large}
A_i^{\mathrm{Big}} \subset \Lambda_N \setminus \Lambda_{N/8} \mbox{ for all } 1\leq i\leq r\,.
\end{equation}
(The geometric setup here is similar to that in the proof of Lemma~\ref{lem-zero-bound-hard-crossing}; see the left picture of Figure~\ref{figure-2.7} for an illustration.) By monotonicity, we see that for each $1\leq i\leq r$
\begin{align*}
\P&(\sum_{v\in A_i} (\langle  \sigma^{\Lambda_N, +, 1}_v\rangle_{\mu^{\Lambda_N, +, 1}} - \langle\sigma^{\Lambda_N, -, 0}_v\rangle_{\mu^{\Lambda_N, -, 0}}) > \theta^*) \\
&\leq \P (\sum_{v\in A_i} (\langle  \sigma^{A_i^{\mathrm{Big}}, +, 1}_v\rangle_{\mu^{A_i^{\mathrm{Big}}, +, 1}} - \langle\sigma^{A_i^{\mathrm{Big}}, -, 0}_v\rangle_{\mu^{A_i^{\mathrm{Big}}, -, 0}}) > \theta^*)\\
& = \P (m^{\Lambda_{N/8}, +, t^*} -m^{\Lambda_{N/8}, -, t^*} > \theta^*) \leq 1/2r\,,
\end{align*}
where the equality holds due to \eqref{eq-A-i-large} and $\Delta' = t^* \Delta$ (note that $h^{(t)}_v = h_v + \Delta'$ for $v\in \Lambda_N\setminus \Lambda_{N/8}$ and for \emph{all} $0\leq t\leq 1$), and in addition the last inequality holds due to \eqref{eq-Delta'}. Thus, a simple union bound gives that the event $\{\sum_{v\in \mathcal A_{N/2}} (\langle  \sigma^{\Lambda_N, +, 1}_v\rangle_{\mu^{\Lambda_N, +, 1}} - \langle\sigma^{\Lambda_N, -, 0}_v\rangle_{\mu^{\Lambda_N, -, 0}}) \leq r \theta^*\}$  contains an event $\mathcal E_{\mathcal A_{N/2}}$ which is measurable with respect to $\{h_v: v\not\in \Lambda_{N/8}\}$ such that
\begin{equation}\label{eq-Step-2}
\P (\mathcal E_{\mathcal A_{N/2}})  \geq 1/2 \,.
\end{equation}

Furthermore, let $\mathcal T = \{1\leq t \leq 1: m^{\Lambda_{N/8}, +, t} - m^{\Lambda_{N/8}, -, t} \geq \theta^*\}$. By \eqref{eq-Delta'} we have $\E |\mathcal T| \geq 1/2r$ where $|\mathcal T|$ is the Lebesgue measure of $\mathcal T$. Since $|\mathcal T| \leq 1$, we have $\P(|\mathcal T| \geq 1/4r) \geq 1/4r$. Therefore,
\begin{equation}\label{eq-lower-bound-T}
\P(\int_0^1 (m^{\Lambda_{N/8}, +, t} - m^{\Lambda_{N/8}, -, t}) dt \geq \theta^* /4r) \geq 1/4r\,.
\end{equation}
Combined with \eqref{eq-Step-2}, this yields that
\begin{equation}\label{eq-Step-2-final}
\P(\mathcal E^\diamond) \geq 1/8r
\end{equation}
where $\mathcal{E}^\diamond$ is the event such that
$$\int_0^1 (m^{\Lambda_{N/8}, +, t} - m^{\Lambda_{N/8}, -, t}) dt \geq \frac{\theta^*}{4r} \geq  (4r^2)^{-1}\sum_{v\in \mathcal A_{N/2}} (\langle  \sigma^{\Lambda_N, +, 1}_v\rangle_{\mu^{\Lambda_N, +, 1}} - \langle\sigma^{\Lambda_N, -, 0}_v\rangle_{\mu^{\Lambda_N, -, 0}})\,.$$

Suppose \eqref{eq-theta*} does not hold. Then by \eqref{eq-Step-1} and the preceding display, the events $\mathcal E^\star$ and $\mathcal E^\diamond$ are mutually exclusive. But by \eqref{eq-prob-E-star} and \eqref{eq-Step-2-final}, we have $\P(\mathcal E^\star) + \P(\mathcal E^\diamond) >1$, arriving at a contradiction.
\end{proof}

\begin{proof}[Proof of Proposition~\ref{prop-crossing-dimension}]
The proof of Proposition~\ref{prop-crossing-dimension} at this point is highly similar to that of Proposition~\ref{prop-zero-crossing-dimension}. As a result, we only provide a sketch emphasizing the additional subtleties.

Let $\pi$ be an arbitrary monotone coupling of $\mu^{\Lambda_N, \pm}$ and let $\mathcal C^{\Lambda_N} = \mathcal C^{\Lambda_N, \pi}$ be defined as in \eqref{eq-def-mathcal-C}.

For any rectangle $A\subset \mathbb R^2$ (whose sides are not necessarily parallel to the axes), recall that $\ell_A$ is the length of the longer side and  $ A^{\mathrm{Large}}$ is the square box concentric with $A$ and of side length $32\ell_A$. In addition,  the aspect ratio of $A$ is the ratio between the lengths of the longer and shorter sides. Consider an arbitrary rectangle $A$ with aspect ratio at least $a = 100$. For a (random) set $\mathcal C \subset \mathbb Z^2$, we continue to use $\mathrm{Cross}(A, \mathcal C)$ to denote the event that there exists a path $v_0, \ldots, v_k \in A \cap \mathcal C$ connecting the two shorter sides of $A$.  For \emph{any} monotone coupling $\pi^{A^{\mathrm{Large}}}$ of $\mu^{A^{\mathrm{Large}}, \pm}$ (below we denote $\mathcal C^{A^{\mathrm{Large}}} = \{v\in A^{\mathrm{Large}}: \sigma^{A^{\mathrm{Large}}, +} > \sigma^{A^{\mathrm{Large}}, -}\}$ under $\pi^{A^{\mathrm{Large}}}$), we can adapt the proof of \eqref{eq-zero-cross-A-hard} and deduce that (write $N' = \min\{2^n: 2^{n+2} \geq \ell_A\}$, and recall $\mathcal E^+$ as in Lemma~\ref{lem-crossing-key})
\begin{align*}
\P \otimes \mu^{\Lambda_{N'}, +}(\mathcal E^+_{N'}) & \geq  1 - 4 (1 - \P\otimes \pi^{A^{\mathrm{Large}}}(\mathrm{Cross}(A, \mathcal V^{ \sigma^{A^{\mathrm{Large}}, +}, +})))\\
&\geq  1 - 4 (1 - \P\otimes \pi^{A^{\mathrm{Large}}}(\mathrm{Cross}(A, \mathcal C^{A^{\mathrm{Large}}})))\,,
\end{align*}
where the second inequality follows from the fact that $\mathrm{Cross}(A, \mathcal C^{A^{\mathrm{Large}}}) \subset \mathrm{Cross}(A, \mathcal V^{ \sigma^{A^{\mathrm{Large}}, +}, +})$.
In addition, by a similar derivation of \eqref{eq-Markov-inequality},
\begin{align*}
\P \otimes \pi^{A^{\mathrm{Large}}}( \mathrm{Cross}(A, \mathcal C^{A^{\mathrm{Large}}})) &\leq \P \otimes \pi^{A^{\mathrm{Large}}} (\#(\mathcal C^{A^{\mathrm{Large}}} \cap A) \geq \ell_A/2) \\
&\leq  \frac{1}{2}(1+\P (\mbox{$\sum_{v\in \Lambda_{N'/8}}$} (\langle  \sigma^{\Lambda_{N'}, +}_v\rangle_{\mu^{\Lambda_{N'}, +}} - \langle\sigma^{\Lambda_{N'}, -}_v\rangle_{\mu^{\Lambda_{N'}, -}}) > 10^{-3} N'))\,.
\end{align*}
Therefore, by Lemma~\ref{lem-crossing-key},
\begin{equation}\label{eq-crossing-prob}
\P\otimes \pi^{A^{\mathrm{Large}}}(\mathrm{Cross}(A, \mathcal C^{A^{\mathrm{Large}}})) \leq 1- \delta \mbox{ where } \delta=  \delta(\epsilon, \beta) >0\,.
\end{equation}
It is crucial that \eqref{eq-crossing-prob} holds uniformly for all possible monotone couplings $\pi^{A^{\mathrm{Large}}}$. Note that the probability for $\mathrm{Cross}(A, \mathcal C^{\Lambda_N, \pi})$ could potentially depend on the location of $A$, either due to different influences from the boundary at different locations or different coupling mechanisms  chosen at different location. However, thanks to \eqref{eq-crossing-prob}, all these probabilities have a uniform upper bound which is strictly less than 1. In addition, by monotonicity of the Ising model, for a collection of rectangles that are well-separated, the corresponding crossing events can be dominated by independent events which have probabilities strictly less than 1.
Next, we complete the proof of Proposition~\ref{prop-crossing-dimension} by utilizing this intuition.  For any $k\geq 1$ and any rectangles $A_1, \ldots, A_k \subseteq \{v\in \mathbb R^2: |v|_\infty\leq N/2\}$ with aspect ratios at least $a$ such that  (a) $\ell_0\leq \ell_{A_i} \leq N/32$ for all $1\leq i\leq k$ and (b) $A^{\mathrm{Large}}_1, \ldots, A^{\mathrm{Large}}_k$ are disjoint, we see that under any coupling $\pi$ of $\mu^{\Lambda_N, \pm}$, there exist sets $\mathcal C^{A_i^{\mathrm{Large}}}$ such that
 \begin{itemize}
\item $\mathcal C^{A_i^{\mathrm{Large}}}$ is sampled according to \emph{some} monotone coupling of $\mu^{A_i^{\mathrm{Large}}, \pm}$.
\item $\mathcal C^{\Lambda_N, \pi} \cap A_i \subset \mathcal C^{A_i^{\mathrm{Large}}} \cap A_i$ (by monotonicity of Ising model with respect to boundary conditions).
\item $\mu^{A_i^{\mathrm{Large}}, \pm}$'s are mutually independent (as they only depend on $\{h_v: v\in A_i^{\mathrm{Large}}\}$ respectively).
 \end{itemize}
Therefore, by \eqref{eq-crossing-prob},
$$\P\otimes \pi( \cap_{i=1}^k \mathrm{Cross}(A_i, \mathcal C^{A_i^{\mathrm{Large}}})) \leq (1- \delta)^k\,.$$
This proves an analogue of Lemma~\ref{lem-zero-assumption}, which verifies the hypothesis required in order to apply \cite{AB99}. The remaining proof is merely an adaption of Proposition~\ref{prop-zero-crossing-dimension} and thus we omit further details.
\end{proof}

\subsection{Admissible coupling and adaptive admissible coupling}\label{sec-admissible-coupling}

In Sections~\ref{sec-admissible-coupling} and \ref{sec-multi-scale}, we wish to prove an analogue of Lemma~\ref{lem-zero-m-star}. In the case for $T>0$, it seems quite a bit more challenging as the choice of the coupling for various Ising measures plays a role, which seems to be subtle in light of Remark~\ref{remark-sad} below. To address the issue, we present a general class of couplings for various Ising measures (i.e., adaptive admissible couplings) in this section. In Section~\ref{sec-contruction-adaptive-coupling}, we present a particular construction of adaptive admissible coupling, which is suited for the multi-scale analysis (the multi-scale analysis is a more complicated version of the proof for Lemma~\ref{lem-zero-m-star}) presented in Section~\ref{sec-analysis-coupling}.

For $k\geq 1$, we consider \emph{deterministic} boundary conditions and external fields $(\tau^{(i)}, \{h^{(i)}_v: v\in \Lambda\})$ where $\tau^{(i)}\in \{-1, 1\}^{\partial \Lambda}$ for $1\leq i\leq k$ (these will be fixed throughout this section). We define the partial order $\prec$ by
\begin{equation}
i \prec j \mbox{ if } \tau^{(i)} \leq \tau^{(j)} \mbox{ and }h^{(i)} \leq h^{(j)}\,.
\end{equation}
We say that $(\sigma^{(1)}, \ldots, \sigma^{(k)})$  (for $\sigma^{(1)}, \ldots, \sigma^{(k)} \in \{-1, 1\}^\Lambda$) is an admissible configuration if
$
\sigma^{(i)} \leq \sigma^{(j)} \mbox{ for all } i \prec j$. Denote by $\Sigma_k$ the collection of all admissible configurations. For $A\subset \Lambda$, write $(\sigma^{(1)}, \ldots, \sigma^{(k)})_A$ for the restriction of $(\sigma^{(1)}, \ldots, \sigma^{(k)})$ on $A$.
\begin{definition}\label{def-admissible}
For each $1\leq i\leq k$, let  $\mu^{(i)}$ be the Ising measure on $\Lambda$ with boundary condition $\tau^{(i)}$ and external field $h^{(i)}$. We say that a measure $\pi$ is an admissible coupling of $\mu^{(1)}, \ldots, \mu^{(k)}$ if $\pi$ is supported on $\Sigma_k$ and its marginal distributions agree with $\mu^{(i)}$'s.
\end{definition}

\begin{remark}\label{remark-sad}
Ideally, it would be great if there would exist an admissible coupling $\pi$ which satisfies the Markov field property. Or, it would also be great if there would exist an admissible coupling $\pi$ which satisfies a weak version of Markov field property, such that  for any $\Gamma \subset \Lambda$ the measure
$\pi(\sigma^{(i)}_{S_\Gamma} \in \cdot \mid (\sigma^{(1)}, \ldots, \sigma^{(k)})_\Gamma)$ is the Ising measure on $S_\Gamma$ with boundary condition $\sigma^{(i)}_{\partial S_\Gamma}$ and external field $\{h^{(i)}_v: v\in S_\Gamma\}$. However, such coupling does not exist as we can see from the following simple example. Let us consider Ising measures on a line segment with no external field and plus/minus boundary conditions on one end (denoted as $u$). Suppose that there exists an admissible coupling $\pi$  (in this case a monotone coupling) with weak Markov field property. Then  conditioned on the event that the two spins disagree at the other end of the line (denoted as $v$), we claim that the spins from the two Ising measures have to disagree on every vertex on the line,  thereby violating the weak Markov property.  In order to verify the claim, we suppose the claim fails and let $w$ be the first vertex (from $u$) where the two spins agree with each other. Conditioned on spins from $u$ to $w$, the two marginals at $v$ are the same (by the weak Markov property) and thus have to agree in a monotone coupling.
\end{remark}

In light of Remark~\ref{remark-sad}, we will seek for admissible couplings with a desirable property even weaker than the weak Markov field property. To this end, we will explore the spins using certain ``adaptive'' algorithm and then wish to argue that the marginal measures on the unexplored region remain to be Ising measures. This motivates us to consider the \emph{adaptive admissible coupling} (see Definition~\ref{def-adaptive-admissible-coupling} below). Let $\Xi_k = \{(\sigma^{(1)}, \ldots, \sigma^{(k)}) \in \{-1, 1\}^k : \sigma^{(i)} \leq \sigma^{(j)} \mbox{ for all } i \prec j\}$. For $\theta_1, \ldots, \theta_k$ which are measures on $\{-1, 1\}$, we say that $\theta_1, \ldots, \theta_k$ are admissible if $\theta_i(1)\leq \theta_j(1)$ for all $i \prec j$. In this case, let $\theta$ be the monotone coupling of $\theta_1, \ldots, \theta_k$. That is, $\theta$ is the joint measure of $(\sigma_1, \ldots, \sigma_k)$, which is defined in terms of a uniform variable $U$ on $[0, 1]$ such that
$$\sigma_i =  -1 \mbox{ if and only if } U \leq 1-\theta_i(1)\,.$$
Clearly, $\theta$ is supported on $\Xi_k$ and its marginals are $\theta_1, \ldots, \theta_k$. In addition, $\theta$ is consistent, i.e.,
\begin{equation}\label{eq-theta-consistent}
\mbox{The projection of }\theta \mbox{  onto the first $(k-1)$ spins is the monotone coupling for } \theta_1, \ldots, \theta_{k-1}\,.
\end{equation}
In order to define adaptive admissible couplings, we make use of exploration procedures. An exploration procedure can be encoded by a family of deterministic maps $\{f_V: V\subset \Lambda, V\neq \Lambda\}$ where $f_V$ is a mapping that maps an admissible configuration on $V$ to a vertex in  $\Lambda \setminus V$. That is to say, if we have explored a set $V\subset \Lambda$ and the spin configuration on $V$ is given by $(\sigma^{(1)}, \ldots, \sigma^{(k)})_V$, then the next vertex we will explore is $f_V((\sigma^{(1)}, \ldots, \sigma^{(k)})_V)$.
\begin{definition}\label{def-adaptive-admissible-coupling}
For each exploration procedure $\{f_V\}$, we associate an admissible coupling in the following manner. Let $\mathcal V_0 = \emptyset$. For $t\geq 1$, let $v_t  = f_{\mathcal V_{t-1}}((\sigma^{(1)}, \ldots, \sigma^{(k)})_{\mathcal V_{t-1}})$. Let $\mathcal V_t = \mathcal V_{t-1} \cup \{v_t\}$. Quenched on the realization of $\{\mathcal V_{t-1}, (\sigma^{(1)}, \ldots, \sigma^{(k)})_{\mathcal V_{t-1}}\}$, for $1\leq i\leq k$ let $\theta_i^{(t)}(\pm 1) = \mu^{(i)}(\sigma^{(i)}_{v_{t}} = \pm1 \mid  \sigma^{(i)}_{\mathcal V_{t-1}})$. Let $\theta^{(t)}$ be the monotone coupling of $\theta^{(t)}_1, \ldots, \theta^{(t)}_k$, and we sample $(\sigma^{(1)}, \ldots, \sigma^{(k)})_{v_t}$ according to $\theta^{(t)}$. We repeat this procedure until $t = \#\Lambda$. We let $\pi$ be the measure on $(\sigma^{(1)}, \ldots, \sigma^{(k)})$ at the end of the procedure. In addition, we say that a random set $\mathcal V$ is a \emph{stopping set} if $\{\mathcal V = \mathcal V_t = V_t\}$ (for any deterministic $V_t \subset \Lambda$) is measurable with respect to $\{(\sigma^{(1)}, \ldots, \sigma^{(k)})_{V_{t}})\}$.
\end{definition}
\begin{remark}
In the study of spin models, it is common to use an exploration procedure to discover certain observables (such as interfaces) associated with spin configurations. Often times, an instance of spin configurations is sampled a priori (which is usually sampled according to a Gibbs measure) and then the exploration procedure is performed on this instance. This is different from Definition~\ref{def-adaptive-admissible-coupling}, where the spin configuration is sampled as the exploration procedure evolves and more importantly the measure on spin configurations depends on the exploration procedure.
\end{remark}

\begin{lemma}\label{lem-adaptive-admissible-coupling}
For each exploration procedure, the measure $\pi$ given in Definition~\ref{def-adaptive-admissible-coupling} is a well-defined admissible coupling. In addition, for any stopping set $\mathcal V$, given the realization of $\mathcal V$ and $(\sigma^{(1)}, \ldots, \sigma^{(k)})_{\mathcal V}$, the conditional measure of $\pi$ restricted on $\mathcal V^c$ has marginals corresponding to Ising measures on $\mathcal V^c$ with boundary condition $\sigma^{(i)}_{\partial \mathcal V^c}$ and external field $\{h^{(i)}_v: v\in \mathcal V^c\}$.
\end{lemma}
\begin{proof}
The measure $\pi$ is well-defined since we can inductively verify that for $t=0, 1, 2,\ldots$, the sequence $\theta^{(t)}_1, \ldots, \theta^{(t)}_k$ is admissible and thus $(\sigma^{(1)}, \ldots, \sigma^{(k)})_{\mathcal V_{t+1}}$ is admissible. To prove the second part of the statement, it suffices to show that for each $1\leq i\leq k$ and $1\leq t\leq \#\Lambda$,
\begin{equation}\label{eq-admissible-property}
\pi(\sigma^{(i)}_{\Lambda \setminus V_{t-1}} \in \cdot \mid  (\sigma^{(1)}, \ldots, \sigma^{(k)})_{\mathcal V_{t-1}}, \mathcal V_{t-1} = V_{t-1}) = \mu^{(i)}(\sigma^{(i)}_{\Lambda \setminus V_{t-1}} \in \cdot  \mid  \sigma^{(i)}_{V_{t-1}})\,.
\end{equation}
We prove \eqref{eq-admissible-property} by induction for $t = \# \Lambda, \ldots, 1$. It is obvious from Definition~\ref{def-adaptive-admissible-coupling} that \eqref{eq-admissible-property} holds for $t = \# \Lambda$. Suppose \eqref{eq-admissible-property} holds for $t$, we then deduce for $t-1$ that
\begin{align*}
&\pi(\sigma^{(i)}_{\Lambda \setminus (V_{t-2} \cup \{v_{t-1}\})} \in \cdot,  \sigma^{(i)}_{v_{t-1}}  = \pm 1\mid  (\sigma^{(1)}, \ldots, \sigma^{(k)})_{\mathcal V_{t-2}}, \mathcal V_{t-2} = V_{t-2}) \\
&= \mu^{(i)}(\sigma^{(i)}_{v_{t-1}} = \pm1 \mid  \sigma^{(i)}_{V_{t-2}}) \times  \mu^{(i)}(\sigma^{(i)}_{\Lambda \setminus (V_{t-2} \cup \{v_{t-1}\})} \in \cdot \mid  \sigma^{(i)}_{V_{t-2}}, \sigma^{(i)}_{v_{t-1}} = \pm1)\,.
\end{align*}
This implies that $\pi(\sigma^{(i)}_{\Lambda \setminus V_{t-2}}  \in \cdot\mid  (\sigma^{(1)}, \ldots, \sigma^{(k)})_{\mathcal V_{t-2}}, \mathcal V_{t-2} = V_{t-2}) = \mu^{(i)}(\sigma^{(i)}_{\Lambda \setminus V_{t-2}} \in \cdot  \mid  \sigma^{(i)}_{V_{t-2}})$, thereby completing the proof by induction.
\end{proof}
In what follows, we refer to $\pi$ as in Definition~\ref{def-adaptive-admissible-coupling} as an adaptive admissible coupling. In addition, we will always define adaptive admissible couplings by presenting an exploration procedure and then consider the associated admissible coupling given in Definition~\ref{def-adaptive-admissible-coupling}. For convenience of exposition, we usually describe an exploration procedure in words rather than specifying the maps $\{f_V\}$.

\subsection{A multi-scale analysis via another perturbation argument}\label{sec-multi-scale}

Let $\alpha>1$ be as in Proposition~\ref{prop-crossing-dimension}. Let $\sqrt{1/\alpha} < \alpha' < 1$.
Let $N_0 = N_0(\epsilon, \beta)$  be a large number to be chosen.
For each $N \geq N_0$ (of the form $4^n$), set  $\Delta = \Delta(N) = N^{-\alpha(\alpha')^2}$. In the rest of the paper, we consider the following perturbation:
\begin{equation}\label{eq-def-tilde-h-again}
\tilde h^{(N)}_v = \begin{cases}
h_v + \Delta, &\mbox{ for } v\in \Lambda_N \setminus \Lambda_{N/4}\,,\\
h_v, &\mbox{ for } v\in \Lambda_{N/4}\,.
\end{cases}
\end{equation}
We denote by $\tilde \mu^{\Lambda_{N}, \pm}$ the Ising measures on $\Lambda_{N}$ with respect to plus/minus boundary conditions and external field $\{\tilde h^{(N)}_v: v\in \Lambda_{N}\}$, and denote by $\tilde\sigma^{\Lambda_{N}, \pm}$ the spins sampled according to $\tilde \mu^{\Lambda_{N}, \pm}$. In this whole section except in \eqref{eq-def-E} and \eqref{eq-prob-D-N}, we will quench on the realization of $\{h_v\}$ and thus the external field is viewed as deterministic.

\subsubsection{A construction of an adaptive admissible coupling}\label{sec-contruction-adaptive-coupling}
We will define the following adaptive admissible coupling $\pi_{\Lambda_N}$ for $\mu^{\Lambda_{N}, \pm}$ and $\tilde \mu^{\Lambda_{N}, \pm}$.
According to Definition~\ref{def-adaptive-admissible-coupling}, in order to specify $\pi_{\Lambda_N}$, we only need to specify the exploration procedure (i.e., the order of vertices in which we sample the spins), as described as follows.  Throughout the procedure, we let $\mathcal C_*^{\Lambda_N}$ be the collection of vertices $v$ which have been sampled such that $\sigma^{\Lambda_{N}, +}_v > \sigma^{\Lambda_{N}, -}_v$ and $\tilde \sigma_v^{\Lambda_{N}, +} > \tilde \sigma_v^{\Lambda_{N}, -}$. We first sample spins at vertices on $\partial \Lambda_{k}$ for $k= N - 1, N-2, \ldots, \frac{N}{2}$. For vertices on $\partial \Lambda_k$, for concreteness we sample in clockwise order starting from the right top corner. Next,
let $K = \lfloor N^{\alpha' \alpha} \rfloor$ and $\ell = \lfloor \frac{1}{4}N^{1-\alpha'}\rfloor$. A comment on the order of the scales chosen: the exploration procedure below contains $\ell$ phases, and in every phase we consider an annulus where the inner and outer boundaries have Euclidean distance $N^{\alpha'}$ and thus by Proposition~\ref{prop-crossing-dimension} typically have intrinsic distance $\geq K \gg N$. This is why we can hope to gain a contraction when comparing the number of disagreements on an annulus to that on its neighboring (larger) annulus (see \eqref{eq-m-j-contraction} below).

We now turn to the description of the exploration procedure. For each $1\leq j\leq \ell$ our construction employs the following procedure which we refer to as Phase $j$ (see Figure~\ref{figure-coupling} for an illustration). Let $N'  = \frac{N}{2} - (j-1) N^{\alpha'}$.
\begin{itemize}
\item We set $A_{j, 0}=\partial  \Lambda_{N'}\cap \mathcal C_*^{\Lambda_{N}}$, $V_{j,0}= \Lambda_N\setminus \Lambda_{N'}$, and for $k=0, 1, \ldots,  K$, we inductively employ the following procedure (which we refer to as stage). At the beginning of Stage $k+1$,  we first set $A_{j, k+1} = \emptyset$ and  $V_{j, k+1} = V_{j, k}$.
\begin{itemize}
\item If $A_{j, k} = \emptyset$ (which we denote as event $\mathcal E_{j, k, \emptyset}$), we sample the unexplored vertices in $\Lambda_{N}$ in a prefixed order (which can be arbitrary) and stop our procedure. Otherwise,
 we explore all the neighbors of $A_{j, k}$ (in a certain prefixed order, which can be arbitrary) which are in $\Lambda_{N'}\setminus  V_{j, k}$ (that is, vertices which have not been explored) and sample the spins at these vertices. We also put these vertices into $V_{j, k+1}.$
\item  If a newly sampled vertex is in $\partial \Lambda_{N' - N^{\alpha'}}$ (we denote this as event $\mathcal E_{j, k, d}$, where the subscript $d$ suggests an event related to the intrinsic distance), we sample the unexplored vertices in $\Lambda_{N}$ in a prefixed order (which can be arbitrary) and stop our procedure.  Otherwise, if
 a newly sampled vertex ends up in $\mathcal C_*^{\Lambda_{N}}$ then we add it to $ A_{j, k+1}$. (For $k\geq 1$, it is clear that $ A_{j, k}$ records all the vertices in $\Lambda_{N'}$ that are of $d_{\mathcal C_*^{\Lambda_{N}}}$-distance $k$ to $\partial \Lambda_{N'}$ and  $V_{j, k}$ records all the explored vertices up to Stage $k$.)
\end{itemize}
\item Sample the unexplored vertices in $\Lambda_{N'} \setminus \Lambda_{N' - N^{\alpha'}}$ in a prefixed order (which can be arbitrary).
\end{itemize}
Finally, if the procedure is not yet stopped after $\ell$ phases, we sample the unexplored vertices in $\Lambda_{N}$ in a prefixed order (which can be arbitrary).
\begin{remark}
(1) Later in the analysis, when we refer to sets such as $A_{j, k}$, $V_{j, k}$ we mean to use their values at the end of our procedure. (2) Note that in the preceding procedure, unless some event of the form $\mathcal E_{j, k, \emptyset}$ or $\mathcal E_{j, k, d}$ occurred, the exploration in all the $\ell$ phases is within $\Lambda_N \setminus \Lambda_{N/4}$.
\end{remark}
\begin{center}
 \begin{figure}[h]
 \vspace{-2cm} \hspace{-1.5cm}
  \includegraphics[width=18cm]{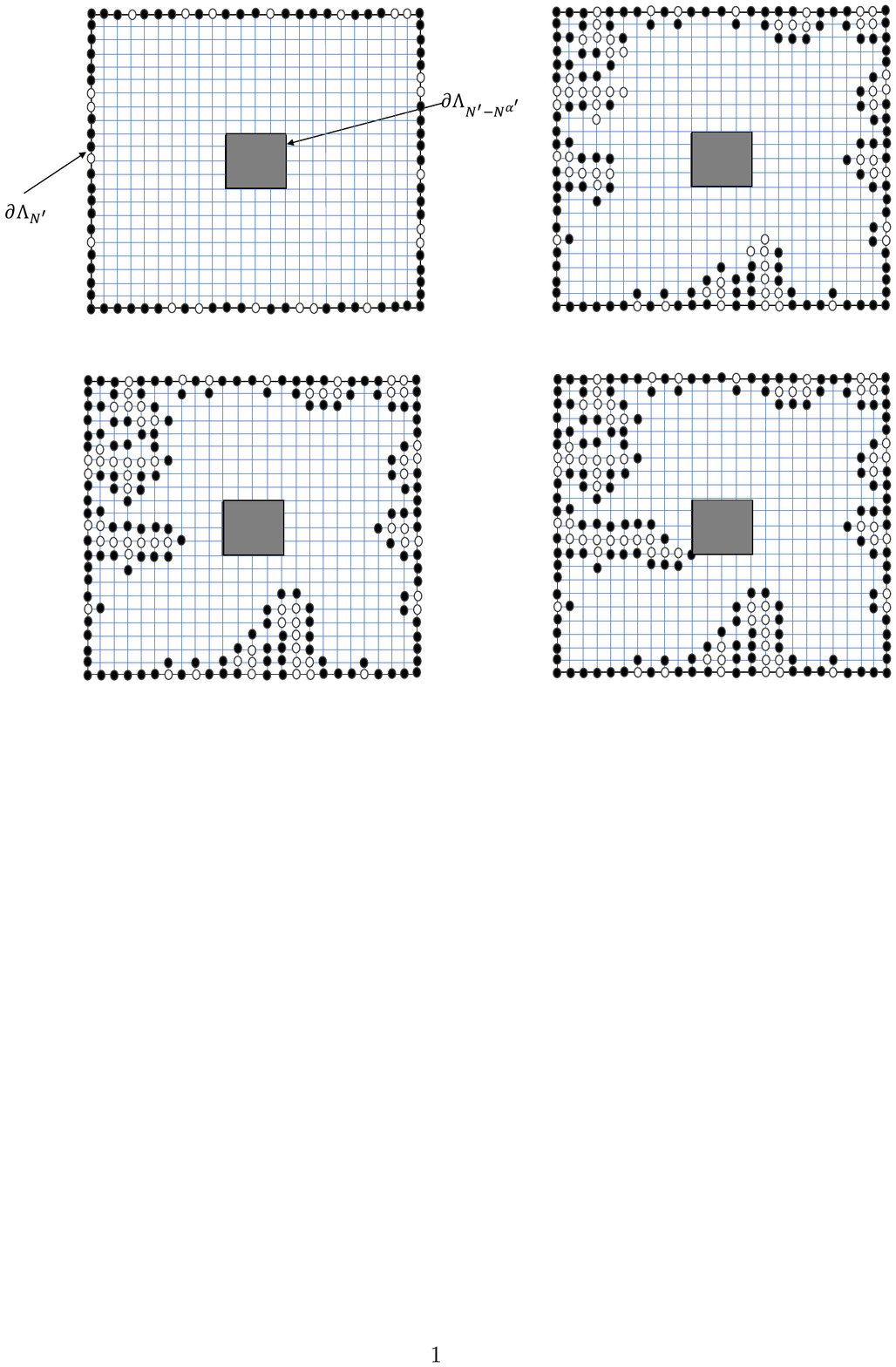}\vspace{-12cm}
  \\  \caption{Illustration for Phase $j$ of the construction in Section~\ref{sec-contruction-adaptive-coupling}. The inside square is $\Lambda_{N' - N^{\alpha'}}$, whose size has been reduced in the picture for better demonstration. On lattice points, empty indicates an unexplored vertex, an open circle indicates a vertex in $\mathcal C_*^{\Lambda_N}$, and a solid disk indicates a vertex not in $\mathcal C_*^{\Lambda_N}$. The top-left illustrates the beginning of Phase $j$, where vertices on $\partial \Lambda_{N'}$ have been explored (vertices outside have been explored too but we did not draw); the top-right illustrates the middle of Phase $j$ (here $k=5$); the bottom-left picture illustrates event $\mathcal E_{j, k, \emptyset}$ (here $k = 8$); the bottom right  event illustrates  $\mathcal E_{j, k, d}$ (here $k = 12$). }\label{figure-coupling}
\end{figure}
\end{center}
\subsubsection{Another perturbation argument}\label{sec-another-perturbation}
We use $\tilde H^{\Lambda_N, \pm}$, $\tilde F^{\Lambda_N, \pm}$, $\tilde \sigma^{\Lambda_N, \pm}$ to denote tilde versions of $H^{\Lambda_N, \pm}$, $F^{\Lambda_N, \pm}$, $\sigma^{\Lambda_N, \pm}$, i.e., defined analogously but with respect to the field $\{\tilde h_v^{(N)}\}$ defined as in \eqref{eq-def-tilde-h-again}. Without further notice, we will always consider measures where we couple all these Ising spins together. Thus, in particular, $\mathcal C^{\Lambda_N}$ and $\tilde {\mathcal C}^{\Lambda_N}$ are defined in the same probability space and
we can then define $\mathcal C_*^{\Lambda_N} = \tilde {\mathcal C}^{\Lambda_N} \cap  {\mathcal C}^{\Lambda_N}$.

We need some preparation before presenting our perturbative analysis. Suppose that $\mathcal V$ is a stopping set (see Definition~\ref{def-adaptive-admissible-coupling}) obtained when constructing $\pi_{\Lambda_N}$ described in Section~\ref{sec-contruction-adaptive-coupling}.  Let $\pi'_{\mathcal V^c}$ be the restriction of $\pi_{\Lambda_N}$ to $\mathcal V^c$.  (We use prime in the notation $\pi'_{\mathcal V^c}$ as we wish to save $\pi_{\mathcal V^c}$ for later use.) By Lemma~\ref{lem-adaptive-admissible-coupling} and our definition of $\pi_{\Lambda_N}$, we see that $\pi'_{\mathcal V^c}$ depends on $(\sigma^{\Lambda_N, \pm})_{\mathcal V}, (\tilde \sigma^{\Lambda_N, \pm})_{\mathcal V}$ only through
$(\sigma^{\Lambda_N, \pm})_{\partial \mathcal V^c}, (\tilde \sigma^{\Lambda_N, \pm})_{\partial \mathcal V^c}$. Thus, we may denote by $(\sigma^{\mathcal V^c, (\sigma^{\Lambda_N, \pm})_{\partial \mathcal V^c}}, \tilde \sigma^{\mathcal V^c, (\tilde \sigma^{\Lambda_N, \pm})_{\partial \mathcal V^c}})$ the spin configurations sampled according to $\pi'_{\mathcal V^c}$ with corresponding boundary conditions on $\partial \mathcal V^c$.
Thus,
\begin{equation}\label{eq-put-admissible-coupling-together}
((\sigma^{\Lambda_N, \pm}_{\mathcal V}, \sigma^{\mathcal V^c, (\sigma^{\Lambda_N, \pm})_{\partial \mathcal V^c}}),   (\tilde \sigma^{\Lambda_N, \pm}_{\mathcal V}, \tilde \sigma^{\mathcal V^c, (\tilde \sigma^{\Lambda_N, \pm})_{\partial \mathcal V^c}})) \mbox{ has law } \pi_{\Lambda_N}\,.
\end{equation}
In what follows, we will mainly consider the measure $\pi'_{\mathcal V^c}$. For clarity of exposition, we quench on the realization of $\mathcal V = V$. Let $S =  V^c$ and $\Gamma = \partial S$ (thus we have $S \subset S_\Gamma$).  Further, we quench on the values of $(\sigma^{\Lambda_N, \pm})_\Gamma,  (\tilde \sigma^{\Lambda_N, \pm})_\Gamma$ by
\begin{equation}\label{eq-quenched-realization}
(\sigma^{\Lambda_N, \pm})_\Gamma = \tau^{\pm}, (\tilde \sigma^{\Lambda_N, \pm})_\Gamma = \tilde \tau^{\pm}\,, \mbox{ where } \tau^\pm, \tilde \tau^\pm \in \{-1, 1\}^\Gamma\,.
\end{equation}

 For $v\in \Gamma$ (in fact, any $v\in \Lambda_N$), by admissibility there are only six possible values for $(\tau^+_v, \tau^-_v, \tilde \tau^+_v, \tilde\tau^-_v)$ as shown in Table \ref{tab:1}. For each such possible spin value, we will define a ``hat'' version $(\hat \tau^+_v, \hat \tau^-_v,  \hat {\tilde \tau}^+_v,  \hat {\tilde \tau}^{-}_v)$, where the definition is given in Table \ref{tab:2}. Note that the hat version is a modification of the original spin value, and we emphasize the change in Table~\ref{tab:2} by circling out the modifications. We will explain why we introduced the hat version of the spin on $\Gamma$ after a number of definitions.  From Tables~\ref{tab:1} and \ref{tab:2}, we see that
\begin{equation}\label{eq-hat-spin-monotonicity}
\hat \tau^+ \geq \hat \tau^- \geq \tau^-, \tilde \tau^+ \geq \hat {\tilde \tau}^+ \geq \hat {\tilde \tau}^-, \hat {\tilde \tau}^+ =  \hat \tau^+ \geq \tau^+, \hat {\tilde \tau}^- = \hat \tau^{-} = \tilde \tau^{-}\,.
\end{equation}
From a notation point of view, despite the fact that $\hat \tau^\pm = \hat {\tilde \tau}^\pm$, we still differentiate these two notations because our mental picture is that the boundary conditions $\hat \tau^{\pm}$ are matched to external field $\{h_v\}$ and the boundary conditions $\hat {\tilde \tau}^\pm$ are matched to external field $\{\tilde h^{(N)}_v\}$.

\begin{table}[h!]
\begin{minipage}{.5\linewidth}
\caption{Original spins on $\Gamma$}
   \centering
    \begin{tabular}{||c c c c c||}
 \hline
type & $\tau_v^{ +}$ & $\tau_v^{-}$ & $\tilde \tau_v^{+}$ & $\tilde\tau_v^{ -}$ \\ [0.5ex]
 \hline\hline
 a. & $- 1$ & $- 1$ & $- 1$ & $- 1$\\
 \hline
 b. & $- 1$ & $- 1$ & +1 & $- 1$ \\
 \hline
 c. & $- 1$ & $- 1$ & +1 & +1 \\
 \hline
 d. & +1 & +1 & +1 & +1 \\
 \hline
 e. & +1 & $- 1$ & +1 & +1 \\
 \hline
 f. & +1 & $- 1$ & +1 & $- 1$ \\ [1ex]
 \hline
 \end{tabular}

\label{tab:1}
\end{minipage}
\begin{minipage}{.55\linewidth}
\caption{The hat version of the spins on $\Gamma$}
\centering
\begin{tabular}{||c c c c c||}
 \hline
type & $\hat \tau_v^{ +}$ & $\hat \tau_v^{ -}$ & $\hat {\tilde \tau}_v^{ +}$ & $\hat {\tilde\tau}_v^{-}$ \\ [0.5ex]
 \hline\hline
 a. & $- 1$ & $- 1$ & $- 1$ & $- 1$\\
 \hline
 $\bullet b$. & $- 1$ & $- 1$ & \circled{$-1$} & $- 1$ \\
 \hline
 $\bullet c$. & \circled{+1} & \circled{$+ 1$} & $ +1$ & $ +1 $ \\
 \hline
 d. & +1 & +1 & +1 & +1 \\
 \hline
 $\bullet e$. & +1 & \circled{+1} & +1 & +1 \\
 \hline
 f. & +1 & $- 1$ & +1 & $- 1$ \\ [1ex]
 \hline

\end{tabular}
\label{tab:2}
\end{minipage}
\end{table}

Recall that $\pi'_S$ is the admissible coupling for Ising measures with boundary conditions and external fields $((\tau^{\pm})_\Gamma, \{h_v\})$, $((\tilde \tau^{ \pm})_\Gamma, \{\tilde h^{(N)}_v\})$, where the order of sampling vertex is given by that of $\pi_{\Lambda_N}$ conditioned on spin configurations on the stopping set $\mathcal V= V$. In addition, we can extend $\pi'_S$ to an adaptive admissible coupling $\pi_S$ for Ising measures with boundary conditions and external fields $((\tau^{\pm})_\Gamma, \{h_v\})$, $((\tilde \tau^{ \pm})_\Gamma, \{\tilde h^{(N)}_v\})$, $((\hat \tau^{\pm})_\Gamma, \{h_v\})$,  $((\hat {\tilde \tau}^{\pm})_\Gamma, \{\tilde h^{(N)}_v\})$, where the order of sampling vertices is determined by the coupling $\pi'_S$. Let $(\sigma^{S, \tau^\pm}, \tilde \sigma^{S, \tilde \tau^\pm}, \sigma^{S, \hat \tau^\pm}, \tilde \sigma^{S, \hat {\tilde {\tau}}^\pm})$ be the spin configuration sampled according to $\pi_S$ (note that we use the tilde symbol on $\sigma$ to emphasize the dependence on the external field $\{\tilde h^{(N)}_v\}$; similarly for $H$ and $F$ below).  By \eqref{eq-theta-consistent}, we see that the projection of $\pi_S$ onto $(\sigma^{S, \tau^\pm}, \tilde \sigma^{S, \tilde \tau^{\pm}})$ has measure $\pi'_S$. As a result, we will simply use $\pi_S$ in what follows.
 We also let $H^{S, \tau^{ \pm}},  \tilde H^{S, \tilde \tau^{ \pm}}, H^{S,  \hat \tau^{ \pm}}, \tilde H^{S, \hat {\tilde \tau}^{\pm}}$ denote Hamiltonians for corresponding Ising spins. Similarly, we denote by $F^{S, \tau^{\pm}}$,  $\tilde F^{S, \tilde \tau^{\pm}}$, $F^{S,  \hat \tau^{\pm}}$, $\tilde F^{S, \hat {\tilde \tau}^{\pm}}$   the log-partition-functions of corresponding Ising measures. Define
$$\mathcal C^{S, \tau^\pm} = \{v\in S: \sigma^{S, \tau^+}_v = 1, \sigma^{S, \tau^{-}}_v = -1\}$$
and similarly define $\tilde {\mathcal C}^{S, \tilde \tau^\pm}, \mathcal C^{S, \hat \tau^\pm}, \tilde {\mathcal C}^{S, \hat {\tilde \tau}^\pm}$. Define $\mathcal C_*^{S, \tau^\pm, \tilde \tau^\pm} = \mathcal C^{S, \tau^\pm} \cap \tilde {\mathcal C}^{S, \tilde \tau^\pm}$ and $\mathcal C_*^{S, \hat \tau^\pm, \hat {\tilde \tau}^\pm} = \mathcal C^{S, \hat \tau^\pm} \cap \tilde {\mathcal C}^{S, \hat {\tilde \tau}^\pm}$.

Now we have necessary notations to explain the reason for introducing the hat version of the spins on $\Gamma$. We wish to bound $\#(\mathcal C_*^{\Lambda_N} \cap S\cap (\Lambda_N\setminus \Lambda_{N/4}))$ in terms of $\#(\mathcal C_*^{\Lambda_N} \cap \Gamma)$. One way to achieve this is to track the change of difference between the log-partition-functions with plus and minus boundary conditions when the external field is perturbed. We see that on the one hand, such change of difference can be  bounded from below in terms of $\#(\mathcal C_*^{\Lambda_N} \cap S \cap (\Lambda_N\setminus \Lambda_{N/4}))$ (see Lemma~\ref{lem-free-energy-differences}); and on the other hand such change can be bounded from above by the number of disagreements for spins on $\Gamma$ with respect to the plus and minus boundary conditions. However, when approaching the upper bound, the spin values of Type b, c, e as in Table \ref{tab:1} will also contribute to the upper bound despite the fact that they do not belong to $\mathcal C_*^{\Lambda_N} \cap \Gamma$. To address this, we introduce the hat version of the spins, which are in agreement  except on  $\mathcal C_*^{\Lambda_N} \cap \Gamma$. A crucial feature  as we will show in Lemma~\ref{lem-hat-version-good}, is that under the admissible coupling $\pi_S$ we have $\mathcal{C}_*^{S, \tau^{\pm}, \tilde \tau^{\pm}}\subset \mathcal C_*^{S, \hat \tau^{\pm}, \hat {\tilde \tau}^{\pm}}$. Therefore, the intended lower bound on the change of log-partition-functions is still valid for the hat version. Another crucial feature of the hat version of the spin is that
\begin{equation}\label{eq-hat-zero}
\begin{split}
&\{v\in \Gamma: \tau^+_v=\tilde \tau^+_v=1,  \tau^-_v= \tilde \tau^-_v = -1\} = \{v\in \Gamma: \hat \tau^+_v=  \hat {\tilde \tau}^+_v = 1, \hat \tau^-_v =\hat {\tilde \tau}^-_v = -1\}\\
&=  \{v\in \Gamma: \hat \tau^+_v = 1, \hat \tau^-_v  = -1\} =  \{v\in \Gamma:  \hat {\tilde \tau}^+_v = 1, \hat {\tilde \tau}^-_v = -1\}\,.
\end{split}
\end{equation}

\begin{lemma}\label{lem-hat-version-good}
Under the admissible coupling $\pi_S$, we have $\mathcal{C}_*^{S, \tau^{\pm}, \tilde \tau^{\pm}}\subset \mathcal C_*^{S, \hat \tau^{\pm}, \hat {\tilde \tau}^{\pm}}$.
\end{lemma}
\begin{proof}
For $u\in \mathcal{C}_*^{S, \tau^{\pm}, \tilde \tau^{\pm}}$, we have $\sigma_u^{S, \tau^+}=\tilde{\sigma}_u^{S, \tilde \tau^+}=1$ and $\sigma_u^{S, \tau^-}=\tilde{\sigma}_u^{S, \tilde \tau^-}=-1$. By \eqref{eq-hat-spin-monotonicity} and the admissible coupling, we see that $\sigma_u^{S, \hat \tau^+} \geq  \sigma_u^{S, \tau^+}=1$; similarly, $\sigma_u^{S, \hat \tau^-} \leq \tilde \sigma_u^{S, \tilde \tau^-}=-1$. So $u\in \mathcal C^{S, \hat \tau^\pm}$. In addition, by \eqref{eq-hat-spin-monotonicity} and the admissible coupling, we see that $\tilde \sigma_u^{S, \hat {\tilde \tau}^+} \geq  \sigma_u^{S, \tau^+}=1$; similarly, $\tilde \sigma_u^{S, \hat {\tilde \tau}^-} = \tilde{\sigma_u}^{S, \tilde \tau^-} = -1$. So $u\in \tilde {\mathcal C}^{S, \hat {\tilde \tau}^\pm}$. Thus, $u\in  \mathcal C_*^{S, \hat \tau^{\pm}, \hat {\tilde \tau}^{\pm}}$ as required.
\end{proof}
\begin{cor}\label{cor-C*-percolates}
Under the admissible coupling $\pi_S$, we have $o\not\in \mathcal{C}_*^{S, \tau^{\pm}, \tilde \tau^{\pm}}$ provided that $\mathcal C_*^{\Lambda_N} \cap \Gamma  = \emptyset$.
\end{cor}
\begin{proof}
If $\mathcal C_*^{\Lambda_N} \cap \Gamma = \emptyset$, we have $\hat \tau^+ = \hat \tau^- = \hat {\tilde \tau}^{+} = \hat {\tilde \tau}^-$, in which case we have $ \mathcal C_*^{S, \hat \tau^{\pm}, \hat {\tilde \tau}^{\pm}} = \emptyset$ and in particular $o\not\in \mathcal C_*^{S, \hat \tau^{\pm}, \hat {\tilde \tau}^{\pm}}$. Combined with Lemma~\ref{lem-hat-version-good}, this completes the proof of the corollary.
\end{proof}
\begin{lemma}\label{lem-free-energy-differences}
We have that
\begin{align}
2\Delta\langle\#(\mathcal C_*^{S, \hat \tau^{\pm}, \hat {\tilde \tau}^{\pm}}\cap (\Lambda_N\setminus \Lambda_{N/4}))\rangle_{\pi_S} & \leq  (\tilde F^{S,  \hat {\tilde \tau}^{+}} - \tilde F^{S, \hat {\tilde \tau}^{-}}) - (F^{S, \hat \tau^{+}} -  F^{S, \hat \tau^{-}})  \label{eq-free-energy-lower}\\
&\leq 16\#\{v\in \Gamma: \hat \tau^+_v=  \hat {\tilde \tau}^+_v = 1, \hat \tau^-_v =\hat {\tilde \tau}^-_v = -1\}\,. \label{eq-free-energy-upper}
 \end{align}
 \end{lemma}
 \begin{proof}
 The proof of the lemma shares some similarity to that of Lemma~\ref{lem-free-energy-differences-2}. However, we give a self-contained proof here in order for clarity of exposition.

 We first prove \eqref{eq-free-energy-upper}. A straightforward computation gives that
 \begin{align*}
\tilde F^{S,  \hat {\tilde \tau}^{+}} - \tilde F^{S, \hat {\tilde \tau}^{-}}
  = \frac{1}{\beta}\log \frac{\sum_{\sigma}e^{-\beta  \tilde H^{S, \hat {\tilde \tau}^{+}}(\sigma)}}{\sum_{\sigma}e^{-\beta \tilde H^{S, \hat {\tilde \tau}^{-}}(\sigma) }}
\leq\frac{1}{\beta}\log e^{8\beta \cdot \#\{v\in \Gamma: \hat {\tilde \tau}^+_v \neq \hat {\tilde \tau}^{-}_v\}} \leq 8\cdot \#\{v\in \Gamma: \hat {\tilde \tau}^+_v \neq \hat {\tilde \tau}^{-}_v\}\,.
 \end{align*}
 Similarly, $F^{S,  \hat { \tau}^{+}} - F^{S, \hat {\tau}^{-}}\geq - 8\cdot \#\{v\in \Gamma: \hat { \tau}^+_v \neq \hat {\tau}^{-}_v\}$. Combined with \eqref{eq-hat-zero}, this proves \eqref{eq-free-energy-upper}.

 Now we turn to prove \eqref{eq-free-energy-lower}. We write
 \begin{equation}\label{eq-int}
(\tilde F^{S,  \hat {\tilde \tau}^{+}} - \tilde F^{S, \hat {\tilde \tau}^{-}}) - (F^{S, \hat \tau^{+}} -  F^{S, \hat \tau^{-}})= (\tilde F^{S,  \hat {\tilde \tau}^{+}} -  F^{S, \hat \tau^{+}} ) - (\tilde F^{S, \hat {\tilde \tau}^{-}} -  F^{S, \hat \tau^{-}}).
 \end{equation}
 For $0\leq t\leq 1$, define
 \begin{equation}\label{eq-h-t-4}
 \tilde h^{(t)}_v = \begin{cases}
h_v + t \Delta, &\mbox{ for } v\in \Lambda_N \setminus \Lambda_{N/4}\,,\\
h_v, &\mbox{ for } v\in \Lambda_{N/4}\,.
\end{cases}
\end{equation}
Let $F^{S,  \hat  \tau^{+}, t}$ be the log-partition-function on $S$ with boundary condition  $\hat  \tau^{+}$ (note that $\hat {\tilde \tau}^{+} = \hat \tau^+$ by \eqref{eq-hat-spin-monotonicity}) and external field $\{\tilde h^{(t)}_v\}$. In particular, $F^{S,  \hat \tau^{+},0} = F^{S, \hat \tau^{+}}$ and $F^{S,  \hat  \tau^{+},1}= \tilde F^{S, \hat {\tilde \tau}^{+}}$. Similar notations apply for $F^{S,  \hat  \tau^{-},t}$.
Thus, we get that
 \begin{equation}\label{eq-part}
\tilde F^{S,  \hat {\tilde \tau}^{+}} -  F^{S, \hat \tau^{+}}  =\int_0^1 \frac{d F^{S,  \hat  \tau^{+}, t}}{dt} dt,\quad \quad
 \tilde F^{S, \hat {\tilde \tau}^{-}} -  F^{S, \hat \tau^{-}}=\int_0^1 \frac{d F^{S,  \hat  \tau^{-}, t}}{dt}  dt\,.
 \end{equation}
 Denote by $\sigma^{S, \hat \tau^\pm, t}$ spins sampled according to Ising measures with boundary conditions $\hat \tau^{\pm}$ and external field  $\{\tilde h^{(t)}\}$.  In addition, for any fixed $t$, we let $\pi_{S, t}$ be the admissible coupling extended from $\pi_S$ by also incorporating the spins $\sigma^{S, \hat \tau^{\pm}, t}$ (again, the order of sampling vertex is given by that of $\pi_S$). Therefore, we see
 $$\frac{d F^{S,  \hat \tau^{+}, t}}{dt}  = \Delta \sum_{v \in S\cap (\Lambda_N \setminus \Lambda_{N/4})}\langle \sigma_v^{S, \hat \tau^+, t}\rangle_{\pi_{S, t}} \mbox{ and } \frac{d F^{S,  \hat  \tau^{-}, t}}{dt}  = \Delta \sum_{v \in S\cap (\Lambda_N \setminus \Lambda_{N/4})}\langle \sigma_v^{S, \hat \tau^-, t}\rangle_{\pi_{S, t}}\,.$$
Combined with \eqref{eq-part} and \eqref{eq-int}, it yields that
\begin{equation}\label{eq-C-*-t}
(F^{S,  \hat {\tilde \tau}^{+}} - F^{S, \hat {\tilde \tau}^{-}}) - (F^{S, \hat \tau^{+}} -  F^{S, \hat \tau^{-}}) =2 \int_0^1 \Delta \langle \#\{v\in S\cap (\Lambda_N \setminus \Lambda_{N/4}):  \sigma_v^{S, \hat \tau^+, t} \neq \sigma_v^{S, \hat \tau^-, t}\} \rangle_{\pi_{S, t}} dt\,.
\end{equation}
For any $v\in S$ and $t\in(0,1)$, by admissible coupling we have $ \sigma_v^{S, \hat \tau^+}\leq \sigma_v^{S, \hat \tau^+, t}\leq \tilde \sigma_v^{S, \hat {\tilde \tau}^+}$ and $\sigma_v^{S, \hat \tau^-}\leq \sigma_v^{S, \hat \tau^-, t}\leq \tilde \sigma_v^{S, \hat {\tilde \tau}^-}$.
Therefore, $\{v\in S\cap (\Lambda_N \setminus \Lambda_{N/4}):  \sigma_v^{S, \hat \tau^+, t} \neq \sigma_v^{S, \hat \tau^-, t}\} \supset \mathcal C_*^{S, \hat\tau^{\pm}, \hat {\tilde \tau}^{\pm}} \cap (\Lambda_N \setminus \Lambda_{N/4})$. Combined with \eqref{eq-C-*-t}, this completes the proof of \eqref{eq-free-energy-lower}.
 \end{proof}

\begin{cor}\label{cor-C-*-bound}
Conditioned on the realization of the stopping set $\mathcal V = V$, let $S = V^c$ and $\Gamma = \partial S$. Then we have
$$\Delta\langle\#(\mathcal C_*^{\Lambda_N} \cap S \cap (\Lambda_N \setminus \Lambda_{N/4})) \mid (\sigma^{\Lambda_N, \pm}, \tilde \sigma^{\Lambda_N, \pm})_{V}\rangle_{\pi_{\Lambda_N}} \leq 8 \#\{\Gamma \cap \mathcal C_*^{\Lambda_N}\}\,.$$
\end{cor}
\begin{proof}
Quench on the realization of $(\sigma^{\Lambda_N, \pm}, \tilde \sigma^{\Lambda_N, \pm})_\Gamma$ as in \eqref{eq-quenched-realization}. By
Lemmas~\ref{lem-hat-version-good} and \ref{lem-free-energy-differences},
\begin{align*}
\Delta\langle\#(\mathcal C_*^{S, \tau^{\pm}, \tilde \tau^{\pm}}\cap (\Lambda_N\setminus \Lambda_{N/4}) )\rangle_{\pi_S} &\leq 8\#\{v\in \Gamma: \hat \tau^+_v=  \hat {\tilde \tau}^+_v = 1, \hat \tau^-_v =\hat {\tilde \tau}^-_v = -1\}\\
& = 8\#\{v\in \Gamma: \tau^+_v=\tilde \tau^+_v=1,  \tau^-_v= \tilde \tau^-_v = -1\}\,,
\end{align*}
where the equality follows from \eqref{eq-hat-zero}.
Combined with \eqref{eq-put-admissible-coupling-together}, this completes  the proof of the corollary.
\end{proof}

\subsubsection{Analysis of the adaptive admissible coupling}\label{sec-analysis-coupling}

We now analyze the adaptive admissible coupling $\pi_{\Lambda_N}$.
Recall that $\ell = \lfloor \frac{1}{4} N^{1-\alpha'}\rfloor$ and
 $K =\lfloor N^{\alpha \alpha'} \rfloor$, and define $\mathcal D_N$ to be the event (measurable with respect to the Gaussian field) by
\begin{equation}\label{eq-def-E}
\mathcal D_N = \{\pi_{\Lambda_N} (\min_{1\leq j\leq \ell} d_{\mathcal C^{\Lambda_N}} (\partial \Lambda_{N/2 -  j N^{\alpha'}}, \partial \Lambda_{N/2 - (j-1) N^{\alpha'}}) \leq K) \geq N^{-20}\}\,.
\end{equation}
By Proposition~\ref{prop-crossing-dimension} and a simple Markov's inequality, we see that  for $C=C(\epsilon, \beta)>0$
\begin{equation}\label{eq-prob-D-N}
\P(\mathcal D_N) \leq C N^{-20}\,.
\end{equation}
 In what follows, we quench on the Gaussian field at which $\mathcal D_N$ does not occur.

\begin{lemma}\label{lem-bound-m}
We have that $\pi_{\Lambda_N}(o\in \mathcal C_*^{\Lambda_N}) \leq C N^{-10}$ on $\mathcal D_N^c$, for $C = C(\epsilon, \beta)>0$.
\end{lemma}
\begin{proof}
 For $1\leq j\leq \ell$, $1\leq k\leq K$, let $\mathcal E_{j, k, \emptyset}, \mathcal E_{j, k, d}, V_{j, k}, A_{j, k}$ be defined as in Section~\ref{sec-contruction-adaptive-coupling}.
 For each $1\leq j\leq \ell$, let $\mathcal E_{j, \emptyset} = \cup_{i=1}^{j}\cup_{k=1}^K \mathcal E_{i, k, \emptyset}$ and define
$$m^*_j = \langle \#( \mathcal C_*^{\Lambda_N} \cap (\Lambda_{N/2 - (j-1) N^{\alpha'}} \setminus \Lambda_{N/2 -  j N^{\alpha'}})) \one_{\mathcal E_{j-1, \emptyset}^c} \rangle_{\pi_{\Lambda_N}}\,.$$
By Corollary~\ref{cor-C*-percolates}, it suffices to prove that $m^*_\ell \leq 2 N^{-10}$. To this end, it suffices to prove that for $N\geq N_0 = N_0(\epsilon, \beta)$ (where $N_0$ is to be selected)
\begin{equation}\label{eq-m-j-contraction}
m^*_{j+1} \leq 10^{-3} m^*_j + N^{-10} \mbox{ for all } 1\leq j\leq \ell-1\,.
\end{equation}
Let $\mathcal E_{j, d} = \cup_{i=1}^{j}\cup_{k=1}^K \mathcal E_{i, k, d}$. Since $\pi_{\Lambda_N}(\mathcal E_{j, d}) \leq C N^{-20}$ on $\mathcal D_N^c$,  it suffices to show that
\begin{equation}\label{eq-m-j-to-prove}
\langle \#( \mathcal C_*^{\Lambda_N} \cap (\Lambda_{N/2 - j N^{\alpha'}} \setminus \Lambda_{N/2 -  (j+1) N^{\alpha'}})) \one_{\mathcal E_{j, \emptyset}^c} \one_{{\mathcal E}_{j, d}^c}\rangle_{\pi_{\Lambda_N}} \leq 10^{-3} m^*_j\,.
\end{equation}
Fix $1\leq j\leq \ell$. For $1\leq k\leq K$, write $\mathcal E_{j, \leq k, \emptyset} = \mathcal E_{j-1, \emptyset} \cup\cup_{i=1}^k\mathcal E_{j, i, \emptyset}$ and $\mathcal E_{j, \leq k, d} = \mathcal E_{j-1, d} \cup\cup_{i=1}^k\mathcal E_{j, i, d}$. Thus, we can deduce that
\begin{align*}
&\Delta\langle\#(\mathcal C_*^{\Lambda_N} \cap (\Lambda_{N/2 - j N^{\alpha'}} \setminus \Lambda_{N/2 - (j+1) N^{\alpha'}}))\one_{\mathcal E_{j, \leq k, \emptyset}^c} \one_{{\mathcal E}_{j, \leq k, d}^c} \mid (\sigma^{\Lambda_N, \pm}, \tilde \sigma^{\Lambda_N, \pm})_{V_{j, k}}\rangle_{\pi_{\Lambda_N}}  \\
&=
\one_{\mathcal E_{j, \leq k, \emptyset}^c} \one_{{\mathcal E}_{j, \leq k, d}^c}\Delta\langle\#(\mathcal C_*^{\Lambda_N} \cap (\Lambda_{N/2 - j N^{\alpha'}} \setminus \Lambda_{N/2 - (j+1) N^{\alpha'}})) \mid (\sigma^{\Lambda_N, \pm}, \tilde \sigma^{\Lambda_N, \pm})_{V_{j, k}}\rangle_{\pi_{\Lambda_N}} \\
&\leq 8 \#A_{j, k} \cdot  \one_{\mathcal E_{j, \leq k, \emptyset}^c} \one_{{\mathcal E}_{j, \leq k, d}^c}\,,
\end{align*}
where the equality holds since $\mathcal E_{j, \leq k, \emptyset}$ and ${\mathcal E}_{j, \leq k, d}$ are measurable with respect to $(\sigma^{\Lambda_N, \pm}, \tilde \sigma^{\Lambda_N, \pm})_{V_{j, k}}$, and the inequality is obtained by applying Corollary~\ref{cor-C-*-bound} with $V = V_{j, k}$ (note that $\Lambda_{N/2 - j N^{\alpha'}}  \cap V_{j, k} = \emptyset$ on the event $\mathcal E_{j, \leq k, d}^c$).
Averaging over the conditioning in the preceding display and recalling that $\mathcal E_{j-1, \emptyset} \subset \mathcal E_{j, \leq k, \emptyset} \subset \mathcal E_{j, \emptyset}$ and $\mathcal E_{j, \leq k, d} \subset \mathcal E_{j, d}$, we deduce that
$$\Delta\langle \#( \mathcal C_*^{\Lambda_N} \cap (\Lambda_{N/2 - j N^{\alpha'}} \setminus \Lambda_{N/2 -  (j+1) N^{\alpha'}})) \one_{\mathcal E_{j, \emptyset}^c} \one_{{\mathcal E}_{j, d}^c}\rangle_{\pi_{\Lambda_N}} \leq \langle 8\# A_{j, k} \cdot \one_{\mathcal E_{j-1, \emptyset}^c} \one_{\mathcal E_{j, \leq k, d}^c}\rangle_{\pi_{\Lambda_N}}\,.$$
Since  $\sum_{k=1}^K \# A_{j, k} \cdot \one_{\mathcal E_{j, \leq k, d}^c} \leq \#(\mathcal C_*^{\Lambda_N} \cap (\Lambda_{N/2 - (j-1) N^{\alpha'}} \setminus \Lambda_{N/2 - j N^{\alpha'}}))$, summing the preceding display over $1\leq k\leq K$ yields
 \eqref{eq-m-j-to-prove} (recall that $\Delta K = N^{-\alpha (\alpha')^2}\lfloor N^{\alpha\alpha'} \rfloor  \geq 10^5$ if $N\geq N_0$ for large enough $N_0$). This completes the proof of the lemma.
\end{proof}

\subsection{Proof of Theorem~\ref{thm-main} for $T>0$} \label{sec-proof-main-thm}
We continue to consider $\tilde h^{(N)}$ defined as in \eqref{eq-def-tilde-h-again}, and let $\mu^{\Lambda_N, \pm}, \tilde \mu^{\Lambda_N, \pm}, \pi_{\Lambda_N}$ be defined as in Section~\ref{sec-multi-scale}. For $\delta>0$, let $Q_\delta\subset [-1, 1]$ be the collection of multiples of $\delta$, and for $q\in Q_\delta$ define $\mathcal E_{o, N, q}^*$ to be an event measurable with respect to the Gaussian field by (the tilde symbol only applies on the minus version below)
\begin{equation}\label{eq-def-E*-o-N-pi-q}
\mathcal E_{o,  N,  q}^* =  \{\langle \sigma^{\Lambda_N, +}_o\rangle_{\mu^{\Lambda_N, +}} \geq q + \delta,  \langle \tilde \sigma^{\Lambda_N, -}_o\rangle_{\tilde \mu^{\Lambda_N, -}} \leq q - \delta\}\,.
\end{equation}
By admissibility, on the event $\mathcal E_{o,  N,  q}^*$ we have $\pi_{\Lambda_N}(o\in \mathcal C_*^{\Lambda_N}) \geq \delta$. Combined with Lemma~\ref{lem-bound-m} and \eqref{eq-prob-D-N}, it yields that
\begin{equation}\label{eq-E*-o-N-pi-q-prob}
\P(\mathcal E_{o,  N,  q}^*) = O(N^{-10}/\delta)\,.
\end{equation}
(Throughout, $O(1)$ hides a constant that may depend on $(\epsilon, \beta)$.) Next, we define
\begin{equation}\label{eq-def-E-o-N-pi-q}
\mathcal E_{o,  N,  q} =  \{\langle \sigma^{\Lambda_N, +}_o\rangle_{\mu^{\Lambda_N, +}} \geq q + \delta,  \langle  \sigma^{\Lambda_N, -}_o\rangle_{ \mu^{\Lambda_N, -}} \leq q - \delta\}\,.
\end{equation}
By monotonicity, we thus have
\begin{equation}\label{eq-E-o-N-q-monotonicity}
\mathcal E_{o, N, q} \subset \mathcal E_{o, N', q} \mbox{ and } \mathcal E_{o, N, q}^* \subset \mathcal E_{o, N', q}^* \mbox{ for all } N'\leq N\,.
\end{equation}

\begin{lemma}\label{lem-m-N-bound}
Let $\delta = N^{-3}/3$. There exists $C= C(\epsilon, \beta)>0$ such that  $\P(\mathcal E_{o, N, q}) \leq C N^{-6}$ for all $q\in Q_\delta$.
\end{lemma}
\begin{proof}
While the proof of the lemma is similar to that of Lemma~\ref{lem-zero-m-N-bound}, we nevertheless provide a self-contained proof for clarity of exposition.

For $A\subseteq \mathbb Z^2$, we set $h_A = \sum_{v\in A} h_v$.
Without loss of generality, let us only consider  $N = 4^n$ for some $n\geq 1$, and for $1\leq \ell\leq n$, we define $\{\tilde h^{(4^\ell)}_v: v\in \Lambda_{4^\ell}\}$ as in \eqref{eq-def-tilde-h-again}. Write $\mathfrak A_{\ell} = \Lambda_{4^{\ell}} \setminus \Lambda_{4^{\ell-1}}$.
For $0.9 n\leq \ell \leq n$, let $\mathcal F_\ell = \sigma(h_v: v\in \Lambda_{4^\ell})$ and write
\begin{equation}\label{eq-Gaussian-conditioning}
h_v = (\#\mathfrak A_{\ell})^{-1} h_{\mathfrak A_{\ell}} + g_v \mbox{ for } v\in\mathfrak A_{\ell}\,,
\end{equation}
where $\{g_v: v \in \mathfrak A_{\ell}\}$ is a mean-zero Gaussian process independent of $h_{\mathfrak A_{\ell}}$ and $\{g_v: v \in\mathfrak A_{\ell}\}$ for $0.9 n\leq \ell \leq n$ are mutually independent. Let $\mathcal F'_\ell$ be the $\sigma$-field which contains every event in $\mathcal F_{\ell}$ that is independent of $h_{\mathfrak A_{\ell}}$ (so in particular $\mathcal F_\ell \subset \mathcal F'_{\ell +1} \subset \mathcal F_{\ell+1}$). Write $\mathcal E_* = \cup_{0.9n\leq \ell\leq n} \mathcal E^*_{o, 4^\ell, q}$. By monotonicity of $\langle \sigma^{\Lambda_N, +}_o\rangle_{\mu^{\Lambda_N, +}}$ and  $\langle  \sigma^{\Lambda_N, -}_o\rangle_{ \mu^{\Lambda_N, -}}$ with respect to the external field,  there exists an interval $I_\ell$ measurable with respect to $\mathcal F'_\ell$ such that conditioned on $\mathcal F'_\ell$ we have $\mathcal E_{o, 4^{\ell}, q}$ occurs if and only if $h_{\mathfrak A_{\ell}} \in I_\ell$. Let $I'_\ell$ be the maximal sub-interval of $I_\ell$ which shares the upper endpoint and $|I'_\ell| \leq  \frac{\#\mathfrak A_{\ell}}{4^{ \alpha (\alpha')^2 \ell}}$ (here $|I'_\ell|$ denotes the length of the interval $I'_\ell$).
By definition in \eqref{eq-def-E*-o-N-pi-q} and \eqref{eq-def-tilde-h-again}, we see
 from \eqref{eq-Gaussian-conditioning} that  conditioned on $\mathcal F'_\ell$ we have that $\mathcal E_{o, 4^{\ell}, q}\cap (\mathcal E^*_{o, 4^{\ell}, q})^c$ occurs only if $h_{\mathfrak A_{\ell}} \in I'_\ell$. Thus,
$$\P(\mathcal E_{o, 4^{\ell}, q}\cap (\mathcal E^*_{o, 4^{\ell}, q})^c\mid \mathcal F'_\ell) \leq \P(h_{\mathfrak A_{\ell}}\in I'_\ell)\,, \mbox{ for }0.9 n \leq \ell \leq n \,.$$
Combined with the fact that $\var (h_{\mathfrak A_{\ell}}) = \epsilon^2 \#\mathfrak A_{\ell}$, this gives that for $C=C(\epsilon, \beta)>0$ (whose value may be adjusted below)
$$\P(\mathcal E_{o, 4^{\ell}, q}\cap (\mathcal E^*_{o, 4^{\ell}, q})^c\mid \mathcal F'_\ell) \leq \frac{C}{4^{\ell(\alpha (\alpha')^2-1)}}\,.$$
By \eqref{eq-E-o-N-q-monotonicity}, we have $\mathcal E_{o, N, q} \cap \mathcal E_*^c = \cap_{\ell = 0.9n}^{n} (\mathcal E_{o, 4^{\ell}, t} \cap (\mathcal E^*_{o, 4^{\ell}, q})^c)$. Since $(\mathcal E_{o, 4^{\ell}, t} \cap (\mathcal E^*_{o, 4^{\ell}, q})^c)$ is $\mathcal F_\ell$-measurable (and thus is $\mathcal F'_{\ell+1}$-measurable), we deduce that (recalling $\alpha (\alpha')^2 > 1$)
$$\P(\mathcal E_{o, N, q} \cap \mathcal E_*^c ) \leq C N^{-6}\,.$$
By \eqref{eq-E*-o-N-pi-q-prob},  we have $\P( \mathcal E_*) \leq CN^{-6}$. Combined with the preceding display, this completes the proof of the lemma.
\end{proof}
Define $\mathcal E_{o, N}$ to be an event measurable with respect to the Gaussian field by
\begin{equation}\label{eq-def-E-o-N}
\mathcal E_{o,  N} = \{\langle \sigma^{\Lambda_N, +}_o\rangle_{\mu^{\Lambda_N, +}} -  \langle \sigma^{\Lambda_N, -}_o\rangle_{\mu^{\Lambda_N, -}} \geq N^{-3}\}\,.
\end{equation}
Since $\mathcal E_{o, N} \subset \cup_{q\in Q_\delta} \mathcal E_{o, N, q}$ with $\delta = N^{-3}/3$,
we  get from Lemma~\ref{lem-m-N-bound}  that $\P(\mathcal E_{o, N})  = O(N^{-3})$. Thus,
\begin{align}\label{eq-marginal-bound}
\E (\langle \sigma^{\Lambda_N, +}_o\rangle_{\mu^{\Lambda_N, +}} - \langle \sigma^{\Lambda_N, -}_o\rangle_{\mu^{\Lambda_N, -}}) &\leq 2 \P(\mathcal E_{o, N}) + \E (\one_{\mathcal E_{o, N}^c} (\langle \sigma^{\Lambda_N, +}_o\rangle_{\mu^{\Lambda_N, +}} -  \langle \sigma^{\Lambda_N, -}_o\rangle_{\mu^{\Lambda_N, -}})) \nonumber \\
&= O(N^{-3})\,.
\end{align}
\begin{remark}
In Lemma~\ref{lem-m-N-bound}, we work with $\mathcal E_{o, N, q}$ other than $\mathcal E_{o, N}$, for the reason that we do not have the property that $\mathcal E_{o, N}$ occurs if and only if $h_{\mathfrak A_{\ell+1}}$ is in a certain interval (but the property holds for $\mathcal E_{o, N, q}$).
\end{remark}

In order to prove Theorem~\ref{thm-main}, we will consider a monotone coupling of $\mu^{\Lambda_N, \pm}$ and consider $\mathcal C^{\Lambda_N} = \{v\in \Lambda_N: \sigma^{\Lambda_N, +}_v > \sigma^{\Lambda_N, -}_v\}$. We
 wish to have that $\{o\in \mathcal C^{\Lambda_N}\}$ occurs only if $o$ is connected to $\partial \Lambda_N$ in $\mathcal C^{\Lambda_N}$. However, as we have seen in Remark~\ref{remark-sad}, this property does not hold for all monotone couplings of $\mu^{\Lambda_N, \pm}$ (For instance if we build an adaptive admissible coupling by first sampling the spin at $o$ and then the rest of the spins, then it is possible to get a configuration where the spin disagrees at $o$ but there exists a contour surrounding $o$ where all spins agree on this contour). In order to address this issue, we will construct an adaptive admissible coupling $\bar \pi_{\Lambda_N}$ such that this percolation property holds. Our construction is similar to that in Section~\ref{sec-contruction-adaptive-coupling} in a way that we explore $\mathcal C^{\Lambda_N}$ in a breadth first search order. But our construction now is much simpler as we no longer need to consider multiple phases.

By Definition~\ref{def-adaptive-admissible-coupling}, in order to define $\bar \pi_{\Lambda_N}$ we only need to specify the order of vertices in which we sample the spins, as described as follows.  Throughout the procedure, we let $\mathcal C^{\Lambda_N}$ be the collection of vertices $v$ which have been sampled and satisfy $\sigma^{\Lambda_{N}, +}_v > \sigma^{\Lambda_{N}, -}_v$.
We set $A_{0}=\partial  \Lambda_{N}$ and for $k=0, 1,  2, \ldots$, we inductively employ the following procedure (which we refer to as stage).
\begin{itemize}
\item At stage $k+1$, first set $A_{k+1} = \emptyset$. If $A_{k} = \emptyset$, we sample the unexplored vertices in $\Lambda_{N}$ in an (arbitrary) prefixed order and stop our procedure. Otherwise,
 we explore all the unexplored neighbors of $A_{k}$ (in a certain arbitrary prefixed order) and sample the spins at these vertices.
\item  For each newly sampled vertex, if it is in $\mathcal C^{\Lambda_{N}}$ then we add it to $ A_{k+1}$.
\end{itemize}
\begin{lemma}\label{lem-percolate-property}
Under the coupling $\bar \pi_{\Lambda_N}$, $o\in \mathcal C^{\Lambda_N}$ only if $o$ is connected to $\partial \Lambda_N$ in $\mathcal C^{\Lambda_N}$.
\end{lemma}
\begin{proof}
Let $k_*$ be the first $k$ such that $A_ k = \emptyset$. If $o$ has been explored by the end of Stage~$(k_*-1)$, we see that $o$ is connected to $\partial \Lambda_N$ in $\mathcal C^{\Lambda_N}$. Otherwise, denote $V_{k_*}$ the collection of explored vertices at the end of Stage $(k_*)$. If $o$ was explored in Stage $k_*$, then $o\not\in \mathcal C^{\Lambda_N}$ (since $A_{k_*} = \emptyset$).  If $o$ was not explored by the end of Stage $k^*$, we see that $\sigma^{\Lambda_N, +}$ and $\sigma^{\Lambda_N, -}$ agree on $\partial V_{k_*}^c$, and thus they will have to agree with each other on $V_{k_*}^c$ by Lemma~\ref{lem-adaptive-admissible-coupling} (this is because $\sigma^{\Lambda_N, +}_v$ and $\sigma^{\Lambda_N, -}$ have the same conditional marginal for all $v\in V_{k_*}^c$ and thus have to agree with each other in an admissible coupling). This in particular implies that $o\not\in \mathcal C^{\Lambda_N}$, completing the proof of the lemma.
\end{proof}

\begin{proof}[Proof of Theorem~\ref{thm-main}: $T>0$]
Consider the adaptive admissible coupling $\bar \pi_{\Lambda_N}$. We will use the fact that $ \P\otimes \bar \pi_{\Lambda_N}(v\in \mathcal C^{\Lambda_N}) =  \frac{1}{2}\E(\langle \sigma^{\Lambda_N, +}_v \rangle_{\mu^{\Lambda_N, +}} - \langle \sigma^{\Lambda_N, -}_v \rangle_{\mu^{\Lambda_N,-}}) $ for all $v\in \Lambda_N$.
Let $N_0 = N_0(\epsilon, \beta)$ be chosen later. For any box $B$, recall that $B^{\mathrm{large}}$ is the box  concentric with $B$ of doubled side length. For $B\in \mathcal B(N, N_0)$, we say $B$ is open if $\mathcal C^{\Lambda_N} \cap B \neq \emptyset$. In order to analyze this percolation process, we say a box $B$ is exceptional if $\sum_{v\in B}(\langle \sigma^{B^{\mathrm{large}}, +}_v \rangle_{\mu^{B^{\mathrm{large}}, +}} - \langle \sigma^{B^{\mathrm{large}}, -}_v \rangle_{\mu^{B^{\mathrm{large}},-}}) \geq N_0^{-1/2}$ (so exceptional is a property measurable with respect to $\{h_v: v\in B^{\mathrm{large}}\}$). By  \eqref{eq-marginal-bound} and monotonicity,
$$\P(B \mbox{ is exceptional}) \leq N_0^{1/2} \sum_{v\in B} \E(\langle \sigma^{B^{\mathrm{large}}, +}_v \rangle_{\mu^{B^{\mathrm{large}}, +}} - \langle \sigma^{B^{\mathrm{large}}, -}_v \rangle_{\mu^{B^{\mathrm{large}},-}}) = O(N_0^{-1/2})\,.$$
Recall Definition~\ref{def-percolation-process}. We see that the exceptional boxes on $\mathcal B(N, N_0)$ form a percolation process which satisfies the $(N, N_0, 4, p)$-condition with $p = O(N_0^{-1/2})$.
In addition, for any box $B$ which is not exceptional, denoting by $\mathcal F_B$ the $\sigma$-field generated by spin configurations outside $B^{\mathrm{large}}$, we see from monotonicity that
$$\bar \pi_{\Lambda_N} (B \mbox{ is open } \mid \mathcal F_B) \leq \sum_{v\in B}(\langle \sigma^{B^{\mathrm{large}}, +}_v \rangle_{\mu^{B^{\mathrm{large}}, +}} - \langle \sigma^{B^{\mathrm{large}}, -}_v \rangle_{\mu^{B^{\mathrm{large}},-}}) = O(N_0^{-1/2})\,.$$ Altogether, this implies that the collection of open boxes forms a percolation process which also satisfies the $(N, N_0, 4, p)$-condition with $p = O(N_0^{-1/2})$.
By Lemma~\ref{lem-percolate-property}, in order for $o\in \mathcal C^{\Lambda_N}$, it is necessary that there exists an open lattice animal on $B\in \mathcal B(N, N_0)$ with size at least $\frac{N}{10N_0}$. Now, choosing $N_0$ sufficiently large (so that $p$ is sufficiently small) and applying Lemma~\ref{lem-zero-enhance} yields that $$\P\otimes \bar \pi_{\Lambda_N} (o\in \mathcal C^{\Lambda_N}) \leq c^{-1}e^{-c N} \mbox{ for }c=c(\epsilon, \beta)>0\,,$$ completing the proof of the theorem.
\end{proof}

\noindent {\bf Acknowledgement.} We thank Tom Spencer for introducing the problem to us, thank Steve Lalley for many interesting discussions and thank Subhajit Goswami, Steve Lalley for a careful reading of an earlier version of the manuscript. We also thank Michael Aizenman and Ron Peled for helpful conversations.
We thank two anonymous referees for many helpful suggestions on exposition.

\end{document}